\documentclass[12pt]{amsart}
\usepackage[latin1]{inputenc}
\usepackage{amsbsy}
\usepackage{amscd} 
\usepackage{amsfonts}
\usepackage{amsgen} 
\usepackage{amsmath}
\usepackage{amsopn} 
\usepackage{amssymb}
\usepackage{amstext}
\usepackage{amsthm} 
\usepackage{amsxtra}
\usepackage{stmaryrd}
\usepackage{rotating}
\usepackage{tikz}
\theoremstyle{plain} 
\newtheorem{thm}{Theorem}[section]
\newtheorem{prop}[thm]{Proposition}
\newtheorem{lem}[thm]{Lemma}
\newtheorem{cor}[thm]{Corollary}
\theoremstyle{definition}
\newtheorem{defn}[thm]{Definition}
\newtheorem{rem}[thm]{Remark}
\newtheorem{ex}[thm]{Example}
\newtheorem{quest}[thm]{Question}

\numberwithin{equation}{section}

\renewcommand{\theta}{\vartheta}
\renewcommand{\phi}{\varphi}
\renewcommand{\epsilon}{\varepsilon}
\renewcommand{\subset}{\subseteq}

\newcommand{\N}{\mathbb N}
\newcommand{\Z}{\mathbb Z}

\newcommand{\C}{\mathbb C}

\newcommand{\CC}{\mathcal C}
\newcommand{\DD}{\mathcal D}
\newcommand{\OOO}{\mathcal O}
\newcommand{\HHH}{\mathcal H}
\newcommand{\BBB}{\mathcal B}
\newcommand{\SSS}{\mathcal S}

\usepackage{calc}
\newcounter{PartitionDepth}
\newcounter{PartitionLength}

\newcommand{\partii}[3]{
 \begin{picture}(#3,#1)
 \setcounter{PartitionLength}{#3-#2}
 \setcounter{PartitionDepth}{-1-#1}
 \put(#2,\thePartitionDepth){\line(0,1){#1}}     
 \put(#3,\thePartitionDepth){\line(0,1){#1}}
 \put(#2,\thePartitionDepth){\line(1,0){\thePartitionLength}}
 \end{picture}}

\newcommand{\upparti}[2]{
 \begin{picture}(#2,#1)
 \setcounter{PartitionDepth}{#1}
 \put(#2,0){\line(0,1){#1}}
 \end{picture}}
\newcommand{\uppartii}[3]{
 \begin{picture}(#3,#1)
 \setcounter{PartitionLength}{#3-#2}
 \setcounter{PartitionDepth}{#1}
 \put(#2,0){\line(0,1){#1}}     
 \put(#3,0){\line(0,1){#1}}
 \put(#2,\thePartitionDepth){\line(1,0){\thePartitionLength}}
 \end{picture}}

\newsavebox{\boxpaarpart}
   \savebox{\boxpaarpart}
   { \begin{picture}(0.7,0.5)
     \put(0,0){\line(0,1){0.4}}
     \put(0.5,0){\line(0,1){0.4}}
     \put(0,0.4){\line(1,0){0.5}}
     \end{picture}}
\newsavebox{\boxbaarpart}
   \savebox{\boxbaarpart}
   { \begin{picture}(0.7,0.5)
     \put(0,0){\line(0,1){0.4}}
     \put(0.5,0){\line(0,1){0.4}}
     \put(0,0){\line(1,0){0.5}}
     \end{picture}}
\newsavebox{\boxdreipart}
   \savebox{\boxdreipart}
   { \begin{picture}(1,0.5)
     \put(0,0){\line(0,1){0.4}}
     \put(0.4,0){\line(0,1){0.4}}
     \put(0.8,0){\line(0,1){0.4}}
     \put(0,0.4){\line(1,0){0.8}}
     \end{picture}}
\newsavebox{\boxvierpart}
   \savebox{\boxvierpart}
   { \begin{picture}(1.4,0.5)
     \put(0,0){\line(0,1){0.4}}
     \put(0.4,0){\line(0,1){0.4}}
     \put(0.8,0){\line(0,1){0.4}}
     \put(1.2,0){\line(0,1){0.4}}
     \put(0,0.4){\line(1,0){1.2}}
     \end{picture}}
\newsavebox{\boxvierpartrot}
   \savebox{\boxvierpartrot}
   { \begin{picture}(0.5,0.8)
     \put(0,-0.2){\line(0,1){0.3}}
     \put(0.4,-0.2){\line(0,1){0.3}}
     \put(0,0.1){\line(1,0){0.4}}
     \put(0,0.4){\line(0,1){0.3}}
     \put(0.4,0.4){\line(0,1){0.3}}
     \put(0,0.4){\line(1,0){0.4}}
     \put(0.2,0.1){\line(0,1){0.3}}
     \end{picture}}
\newsavebox{\boxvierpartrotdrei}
   \savebox{\boxvierpartrotdrei}
   { \begin{picture}(1,0.8)
     \put(0,-0.2){\line(0,1){0.3}}
     \put(0.4,-0.2){\line(0,1){0.8}}
     \put(0.8,-0.2){\line(0,1){0.3}}
     \put(0,0.1){\line(1,0){0.8}}
     \end{picture}}
\newsavebox{\boxcrosspart}
   \savebox{\boxcrosspart}
   { \begin{picture}(0.5,0.8)
     \put(0,-0.2){\line(1,2){0.4}}
     \put(0.4,-0.2){\line(-1,2){0.4}}
     \end{picture}}
\newsavebox{\boxhalflibpart}
   \savebox{\boxhalflibpart}
   { \begin{picture}(1,0.8)
     \put(0,-0.2){\line(1,1){0.8}}
     \put(0.8,-0.2){\line(-1,1){0.8}}
     \put(0.4,-0.2){\line(0,1){0.8}}
     \end{picture}}
\newsavebox{\boxpositioner} 
   \savebox{\boxpositioner}
   { \begin{picture}(1.4,0.5)
     \put(0,0){\line(0,1){0.3}}
     \put(0.4,0){\line(0,1){0.6}}
     \put(0.8,0){\line(0,1){0.3}}
     \put(1.2,0){\line(0,1){0.6}}
     \put(0.4,0.6){\line(1,0){0.8}}
     \end{picture}}
\newsavebox{\boxfatcross} 
   \savebox{\boxfatcross}
   { \begin{picture}(1.4,1.3)
     \put(0,-0.2){\line(0,1){0.3}}
     \put(0.4,-0.2){\line(0,1){0.3}}
     \put(0.8,-0.2){\line(0,1){0.3}}
     \put(1.2,-0.2){\line(0,1){0.3}}
     \put(0,0.1){\line(1,0){0.4}}
     \put(0.8,0.1){\line(1,0){0.4}}
     \put(0,0.9){\line(0,1){0.3}}
     \put(0.4,0.9){\line(0,1){0.3}}
     \put(0.8,0.9){\line(0,1){0.3}}
     \put(1.2,0.9){\line(0,1){0.3}}
     \put(0,0.9){\line(1,0){0.4}}
     \put(0.8,0.9){\line(1,0){0.4}}
     \put(0.2,0.1){\line(1,1){0.8}}
     \put(0.2,0.9){\line(1,-1){0.8}}
     \end{picture}}
\newsavebox{\boxprimarypart} 
   \savebox{\boxprimarypart}
   { \begin{picture}(1,1.3)
     \put(0,-0.2){\line(0,1){0.3}}
     \put(0.4,-0.2){\line(0,1){0.3}}
     \put(0.8,-0.2){\line(0,1){0.3}}
     \put(0.4,0.1){\line(1,0){0.4}}
     \put(0,0.9){\line(0,1){0.3}}
     \put(0.4,0.9){\line(0,1){0.3}}
     \put(0.8,0.9){\line(0,1){0.3}}
     \put(0,0.9){\line(1,0){0.4}}
     \put(0,0.1){\line(1,1){0.8}}
     \put(0.2,0.9){\line(1,-2){0.4}}
     \end{picture}}

\newcommand{\primarypart}{\usebox{\boxprimarypart}}

\newcommand{\upsubset}{\begin{rotate}{90}$\subset$\end{rotate}}

\newcommand{\twocol}{\circ\bullet}

\newsavebox{\boxidpartww}
   \savebox{\boxidpartww}
   { \begin{picture}(0.5,1.2)
     \put(0.2,0.15){\line(0,1){0.7}}
     \put(0,-0.2){$\circ$}
     \put(0,0.8){$\circ$}
     \end{picture}}
\newsavebox{\boxidpartbw}
   \savebox{\boxidpartbw}
   { \begin{picture}(0.5,1.2)
     \put(0.2,0.15){\line(0,1){0.7}}
     \put(0,-0.2){$\circ$}
     \put(0,0.8){$\bullet$}
     \end{picture}}
\newsavebox{\boxidpartwb}
   \savebox{\boxidpartwb}
   { \begin{picture}(0.5,1.2)
     \put(0.2,0.15){\line(0,1){0.7}}
     \put(0,-0.2){$\bullet$}
     \put(0,0.8){$\circ$}
     \end{picture}}
\newsavebox{\boxidpartbb}
   \savebox{\boxidpartbb}
   { \begin{picture}(0.5,1.2)
     \put(0.2,0.15){\line(0,1){0.7}}
     \put(0,-0.2){$\bullet$}
     \put(0,0.8){$\bullet$}
     \end{picture}}

\newsavebox{\boxidpartsingletonbb}
   \savebox{\boxidpartsingletonbb}
   { \begin{picture}(0.5,1.2)
     \put(0.2,0.15){\line(0,1){0.2}}
     \put(0.2,0.65){\line(0,1){0.2}}
     \put(0,-0.2){$\bullet$}
     \put(0,0.8){$\bullet$}
     \end{picture}}
\newsavebox{\boxidpartsingletonww}
   \savebox{\boxidpartsingletonww}
   { \begin{picture}(0.5,1.2)
     \put(0.2,0.15){\line(0,1){0.2}}
     \put(0.2,0.65){\line(0,1){0.2}}
     \put(0,-0.2){$\circ$}
     \put(0,0.8){$\circ$}
     \end{picture}}
     
\newsavebox{\boxpaarpartbb}
   \savebox{\boxpaarpartbb}
   { \begin{picture}(1,1)
     \put(0.2,0.2){\line(0,1){0.4}}
     \put(0.7,0.2){\line(0,1){0.4}}
     \put(0.2,0.6){\line(1,0){0.5}}
     \put(0.5,-0.2){$\bullet$}
     \put(0,-0.2){$\bullet$}
     \end{picture}}
\newsavebox{\boxpaarpartww}
   \savebox{\boxpaarpartww}
   { \begin{picture}(1,1)
     \put(0.2,0.2){\line(0,1){0.4}}
     \put(0.7,0.2){\line(0,1){0.4}}
     \put(0.2,0.6){\line(1,0){0.5}}
     \put(0.5,-0.2){$\circ$}
     \put(0,-0.2){$\circ$}
     \end{picture}}
\newsavebox{\boxpaarpartbw}
   \savebox{\boxpaarpartbw}
   { \begin{picture}(1,1)
     \put(0.2,0.2){\line(0,1){0.4}}
     \put(0.7,0.2){\line(0,1){0.4}}
     \put(0.2,0.6){\line(1,0){0.5}}
     \put(0.5,-0.2){$\circ$}
     \put(0,-0.2){$\bullet$}
     \end{picture}}
\newsavebox{\boxpaarpartwb}
   \savebox{\boxpaarpartwb}
   { \begin{picture}(1,1)
     \put(0.2,0.2){\line(0,1){0.4}}
     \put(0.7,0.2){\line(0,1){0.4}}
     \put(0.2,0.6){\line(1,0){0.5}}
     \put(0.5,-0.2){$\bullet$}
     \put(0,-0.2){$\circ$}
     \end{picture}}
\newsavebox{\boxbaarpartbb}
   \savebox{\boxbaarpartbb}
   { \begin{picture}(1,1)
     \put(0.2,-0.1){\line(0,1){0.4}}
     \put(0.7,-0.1){\line(0,1){0.4}}
     \put(0.2,-0.1){\line(1,0){0.5}}
     \put(0.5,0.3){$\bullet$}
     \put(0,0.3){$\bullet$}
     \end{picture}}
\newsavebox{\boxbaarpartww}
   \savebox{\boxbaarpartww}
   { \begin{picture}(1,1)
     \put(0.2,-0.1){\line(0,1){0.4}}
     \put(0.7,-0.1){\line(0,1){0.4}}
     \put(0.2,-0.1){\line(1,0){0.5}}
     \put(0.5,0.3){$\circ$}
     \put(0,0.3){$\circ$}
     \end{picture}}
\newsavebox{\boxbaarpartbw}
   \savebox{\boxbaarpartbw}
   { \begin{picture}(1,1)
     \put(0.2,-0.1){\line(0,1){0.4}}
     \put(0.7,-0.1){\line(0,1){0.4}}
     \put(0.2,-0.1){\line(1,0){0.5}}
     \put(0.5,0.3){$\circ$}
     \put(0,0.3){$\bullet$}
     \end{picture}}
\newsavebox{\boxbaarpartwb}
   \savebox{\boxbaarpartwb}
   { \begin{picture}(1,1)
     \put(0.2,-0.1){\line(0,1){0.4}}
     \put(0.7,-0.1){\line(0,1){0.4}}
     \put(0.2,-0.1){\line(1,0){0.5}}
     \put(0.5,0.3){$\bullet$}
     \put(0,0.3){$\circ$}
     \end{picture}}

\newsavebox{\boxcutpaarpartbb}
   \savebox{\boxcutpaarpartbb}
   { \begin{picture}(1,1)
     \put(0,0.2){\line(0,1){0.4}}
     \put(0.8,0.2){\line(0,1){0.4}}
     \put(0,0.6){\line(1,0){0.2}}
     \put(0.6,0.6){\line(1,0){0.2}}
     \put(0.6,-0.2){$\bullet$}
     \put(-0.2,-0.2){$\bullet$}
     \end{picture}}
\newsavebox{\boxcutpaarpartww}
   \savebox{\boxcutpaarpartww}
   { \begin{picture}(1,1)
     \put(0,0.2){\line(0,1){0.4}}
     \put(0.8,0.2){\line(0,1){0.4}}
     \put(0,0.6){\line(1,0){0.2}}
     \put(0.6,0.6){\line(1,0){0.2}}
     \put(0.6,-0.2){$\circ$}
     \put(-0.2,-0.2){$\circ$}
     \end{picture}}
\newsavebox{\boxcutpaarpartbw}
   \savebox{\boxcutpaarpartbw}
   { \begin{picture}(1,1)
     \put(0,0.2){\line(0,1){0.4}}
     \put(0.8,0.2){\line(0,1){0.4}}
     \put(0,0.6){\line(1,0){0.2}}
     \put(0.6,0.6){\line(1,0){0.2}}
     \put(0.6,-0.2){$\circ$}
     \put(-0.2,-0.2){$\bullet$}
     \end{picture}}
\newsavebox{\boxcutpaarpartwb}
   \savebox{\boxcutpaarpartwb}
   { \begin{picture}(1,1)
     \put(0,0.2){\line(0,1){0.4}}
     \put(0.8,0.2){\line(0,1){0.4}}
     \put(0,0.6){\line(1,0){0.2}}
     \put(0.6,0.6){\line(1,0){0.2}}
     \put(0.6,-0.2){$\bullet$}
     \put(-0.2,-0.2){$\circ$}
     \end{picture}}

\newsavebox{\boxsingletonw}
   \savebox{\boxsingletonw}
   { \begin{picture}(0.5, 1)
     \put(0,0.3){$\uparrow$}
     \put(0,-0.2){$\circ$}
     \end{picture}}
\newsavebox{\boxsingletonb}
   \savebox{\boxsingletonb}
   { \begin{picture}(0.5, 1)
     \put(0,0.3){$\uparrow$}
     \put(0,-0.2){$\bullet$}
     \end{picture}}
\newsavebox{\boxdownsingletonw}
   \savebox{\boxdownsingletonw}
   { \begin{picture}(0.5, 1)
     \put(0,0){$\downarrow$}
     \put(0,0.6){$\circ$}
     \end{picture}}
\newsavebox{\boxdownsingletonb}
   \savebox{\boxdownsingletonb}
   { \begin{picture}(0.5, 1)
     \put(0,0){$\downarrow$}
     \put(0,0.6){$\bullet$}
     \end{picture}}

\newsavebox{\boxvierpartwbwb}
   \savebox{\boxvierpartwbwb}
   { \begin{picture}(1.7,1)
     \put(0.2,0.2){\line(0,1){0.4}}
     \put(0.6,0.2){\line(0,1){0.4}}
     \put(1,0.2){\line(0,1){0.4}}
     \put(1.4,0.2){\line(0,1){0.4}}
     \put(0.2,0.6){\line(1,0){1.2}}
     \put(0,-0.2){$\circ$}
     \put(0.4,-0.2){$\bullet$}
     \put(0.8,-0.2){$\circ$}
     \put(1.2,-0.2){$\bullet$}
     \end{picture}}
\newsavebox{\boxvierpartbwbw}
   \savebox{\boxvierpartbwbw}
   { \begin{picture}(1.7,1)
     \put(0.2,0.2){\line(0,1){0.4}}
     \put(0.6,0.2){\line(0,1){0.4}}
     \put(1,0.2){\line(0,1){0.4}}
     \put(1.4,0.2){\line(0,1){0.4}}
     \put(0.2,0.6){\line(1,0){1.2}}
     \put(0,-0.2){$\bullet$}
     \put(0.4,-0.2){$\circ$}
     \put(0.8,-0.2){$\bullet$}
     \put(1.2,-0.2){$\circ$}
     \end{picture}}
\newsavebox{\boxvierpartwwbb}
   \savebox{\boxvierpartwwbb}
   { \begin{picture}(1.7,1)
     \put(0.2,0.2){\line(0,1){0.4}}
     \put(0.6,0.2){\line(0,1){0.4}}
     \put(1,0.2){\line(0,1){0.4}}
     \put(1.4,0.2){\line(0,1){0.4}}
     \put(0.2,0.6){\line(1,0){1.2}}
     \put(0,-0.2){$\circ$}
     \put(0.4,-0.2){$\circ$}
     \put(0.8,-0.2){$\bullet$}
     \put(1.2,-0.2){$\bullet$}
     \end{picture}}
\newsavebox{\boxvierpartwbbw}
   \savebox{\boxvierpartwbbw}
   { \begin{picture}(1.7,1)
     \put(0.2,0.2){\line(0,1){0.4}}
     \put(0.6,0.2){\line(0,1){0.4}}
     \put(1,0.2){\line(0,1){0.4}}
     \put(1.4,0.2){\line(0,1){0.4}}
     \put(0.2,0.6){\line(1,0){1.2}}
     \put(0,-0.2){$\circ$}
     \put(0.4,-0.2){$\bullet$}
     \put(0.8,-0.2){$\bullet$}
     \put(1.2,-0.2){$\circ$}
     \end{picture}}
\newsavebox{\boxvierpartbwwb}
   \savebox{\boxvierpartbwwb}
   { \begin{picture}(1.7,1)
     \put(0.2,0.2){\line(0,1){0.4}}
     \put(0.6,0.2){\line(0,1){0.4}}
     \put(1,0.2){\line(0,1){0.4}}
     \put(1.4,0.2){\line(0,1){0.4}}
     \put(0.2,0.6){\line(1,0){1.2}}
     \put(0,-0.2){$\bullet$}
     \put(0.4,-0.2){$\circ$}
     \put(0.8,-0.2){$\circ$}
     \put(1.2,-0.2){$\bullet$}
     \end{picture}}
\newsavebox{\boxvierpartrotwbwb}
   \savebox{\boxvierpartrotwbwb}
   { \begin{picture}(1,1.2)
     \put(0.2,0.15){\line(0,1){0.2}}
     \put(0.2,0.65){\line(0,1){0.2}}
     \put(0.2,0.65){\line(1,0){0.5}}
     \put(0.2,0.35){\line(1,0){0.5}}
     \put(0.45,0.35){\line(0,1){0.3}}
     \put(0.7,0.15){\line(0,1){0.2}}
     \put(0.7,0.65){\line(0,1){0.2}}
     \put(0,0.8){$\circ$}
     \put(0.5,0.8){$\bullet$}
     \put(0,-0.2){$\circ$}
     \put(0.5,-0.2){$\bullet$}
     \end{picture}}
\newsavebox{\boxvierpartrotbwwb}
   \savebox{\boxvierpartrotbwwb}
   { \begin{picture}(1,1.2)
     \put(0.2,0.15){\line(0,1){0.2}}
     \put(0.2,0.65){\line(0,1){0.2}}
     \put(0.2,0.65){\line(1,0){0.5}}
     \put(0.2,0.35){\line(1,0){0.5}}
     \put(0.45,0.35){\line(0,1){0.3}}
     \put(0.7,0.15){\line(0,1){0.2}}
     \put(0.7,0.65){\line(0,1){0.2}}
     \put(0,0.8){$\bullet$}
     \put(0.5,0.8){$\circ$}
     \put(0,-0.2){$\circ$}
     \put(0.5,-0.2){$\bullet$}
     \end{picture}}
\newsavebox{\boxvierpartrotbwbw}
   \savebox{\boxvierpartrotbwbw}
   { \begin{picture}(1,1.2)
     \put(0.2,0.15){\line(0,1){0.2}}
     \put(0.2,0.65){\line(0,1){0.2}}
     \put(0.2,0.65){\line(1,0){0.5}}
     \put(0.2,0.35){\line(1,0){0.5}}
     \put(0.45,0.35){\line(0,1){0.3}}
     \put(0.7,0.15){\line(0,1){0.2}}
     \put(0.7,0.65){\line(0,1){0.2}}
     \put(0,0.8){$\bullet$}
     \put(0.5,0.8){$\circ$}
     \put(0,-0.2){$\bullet$}
     \put(0.5,-0.2){$\circ$}
     \end{picture}}
\newsavebox{\boxvierpartrotwwww}
   \savebox{\boxvierpartrotwwww}
   { \begin{picture}(1,1.2)
     \put(0.2,0.15){\line(0,1){0.2}}
     \put(0.2,0.65){\line(0,1){0.2}}
     \put(0.2,0.65){\line(1,0){0.5}}
     \put(0.2,0.35){\line(1,0){0.5}}
     \put(0.45,0.35){\line(0,1){0.3}}
     \put(0.7,0.15){\line(0,1){0.2}}
     \put(0.7,0.65){\line(0,1){0.2}}
     \put(0,0.8){$\circ$}
     \put(0.5,0.8){$\circ$}
     \put(0,-0.2){$\circ$}
     \put(0.5,-0.2){$\circ$}
     \end{picture}}
\newsavebox{\boxvierpartrotbbbb}
   \savebox{\boxvierpartrotbbbb}
   { \begin{picture}(1,1.2)
     \put(0.2,0.15){\line(0,1){0.2}}
     \put(0.2,0.65){\line(0,1){0.2}}
     \put(0.2,0.65){\line(1,0){0.5}}
     \put(0.2,0.35){\line(1,0){0.5}}
     \put(0.45,0.35){\line(0,1){0.3}}
     \put(0.7,0.15){\line(0,1){0.2}}
     \put(0.7,0.65){\line(0,1){0.2}}
     \put(0,0.8){$\bullet$}
     \put(0.5,0.8){$\bullet$}
     \put(0,-0.2){$\bullet$}
     \put(0.5,-0.2){$\bullet$}
     \end{picture}}

\newsavebox{\boxdreipartwww}
   \savebox{\boxdreipartwww}
   { \begin{picture}(1.2,1)
     \put(0.2,0.2){\line(0,1){0.4}}
     \put(0.6,0.2){\line(0,1){0.4}}
     \put(1,0.2){\line(0,1){0.4}}
     \put(0.2,0.6){\line(1,0){0.8}}
     \put(0,-0.2){$\circ$}
     \put(0.4,-0.2){$\circ$}
     \put(0.8,-0.2){$\circ$}
     \end{picture}}

\newsavebox{\boxsechspartwbwbwb}
   \savebox{\boxsechspartwbwbwb}
   { \begin{picture}(2.5,1)
     \put(0.2,0.2){\line(0,1){0.4}}
     \put(0.6,0.2){\line(0,1){0.4}}
     \put(1,0.2){\line(0,1){0.4}}
     \put(1.4,0.2){\line(0,1){0.4}}
     \put(1.8,0.2){\line(0,1){0.4}}
     \put(2.2,0.2){\line(0,1){0.4}}
     \put(0.2,0.6){\line(1,0){2}}
     \put(0,-0.2){$\circ$}
     \put(0.4,-0.2){$\bullet$}
     \put(0.8,-0.2){$\circ$}
     \put(1.2,-0.2){$\bullet$}
     \put(1.6,-0.2){$\circ$}
     \put(2.0,-0.2){$\bullet$}
     \end{picture}}
 
\newsavebox{\boxcrosspartwbbw}
   \savebox{\boxcrosspartwbbw}
   { \begin{picture}(1,1.4)
     \put(0.25,0.15){\line(1,2){0.4}}
     \put(0.65,0.15){\line(-1,2){0.4}}
     \put(0,0.9){$\circ$}
     \put(0.5,0.9){$\bullet$}
     \put(0,-0.2){$\bullet$}
     \put(0.5,-0.2){$\circ$}
     \end{picture}}
\newsavebox{\boxcrosspartbwwb}
   \savebox{\boxcrosspartbwwb}
   { \begin{picture}(1,1.4)
     \put(0.25,0.15){\line(1,2){0.4}}
     \put(0.65,0.15){\line(-1,2){0.4}}
     \put(0,0.9){$\bullet$}
     \put(0.5,0.9){$\circ$}
     \put(0,-0.2){$\circ$}
     \put(0.5,-0.2){$\bullet$}
     \end{picture}}
\newsavebox{\boxcrosspartwwww}
   \savebox{\boxcrosspartwwww}
   { \begin{picture}(1,1.4)
     \put(0.25,0.15){\line(1,2){0.4}}
     \put(0.65,0.15){\line(-1,2){0.4}}
     \put(0,0.9){$\circ$}
     \put(0.5,0.9){$\circ$}
     \put(0,-0.2){$\circ$}
     \put(0.5,-0.2){$\circ$}
     \end{picture}}
\newsavebox{\boxcrosspartbbbb}
   \savebox{\boxcrosspartbbbb}
   { \begin{picture}(1,1.4)
     \put(0.25,0.15){\line(1,2){0.4}}
     \put(0.65,0.15){\line(-1,2){0.4}}
     \put(0,0.9){$\bullet$}
     \put(0.5,0.9){$\bullet$}
     \put(0,-0.2){$\bullet$}
     \put(0.5,-0.2){$\bullet$}
     \end{picture}}

\newsavebox{\boxhalflibpartwwwwww}
   \savebox{\boxhalflibpartwwwwww}
   { \begin{picture}(1.5,1.4)
     \put(0.3,0.15){\line(1,1){0.8}}
     \put(1.1,0.15){\line(-1,1){0.8}}
     \put(0.7,0.15){\line(0,1){0.8}}     
     \put(0,0.9){$\circ$}
     \put(0.5,0.9){$\circ$}
     \put(1,0.9){$\circ$}
     \put(0,-0.2){$\circ$}
     \put(0.5,-0.2){$\circ$}
     \put(1,-0.2){$\circ$}     
     \end{picture}}

\newsavebox{\boxpositionerd} 
   \savebox{\boxpositionerd}
   { \begin{picture}(2.6,1.5)
     \put(0.9,0.2){\line(0,1){0.9}}
     \put(2.3,0.2){\line(0,1){0.9}}
     \put(0.9,1.1){\line(1,0){1.4}}
     \put(0.05,-0.05){$^\uparrow$}
     \put(0.2,0.4){\tiny$^{\otimes d}$}
     \put(1.35,-0.05){$^\uparrow$}
     \put(1.5,0.4){\tiny$^{\otimes d}$}
     \put(0,-0.2){$\circ$}
     \put(0.7,-0.2){$\circ$}
     \put(1.3,-0.2){$\bullet$}
     \put(2.1,-0.2){$\bullet$}
     \end{picture}}
\newsavebox{\boxpositionerrpluseins} 
   \savebox{\boxpositionerrpluseins}
   { \begin{picture}(3.9,1.5)
     \put(1.6,0.2){\line(0,1){0.9}}
     \put(3.6,0.2){\line(0,1){0.9}}
     \put(1.6,1.1){\line(1,0){2}}
     \put(0.05,-0.05){$^\uparrow$}
     \put(0.2,0.4){\tiny$^{\otimes r+1}$}
     \put(2.05,-0.05){$^\uparrow$}
     \put(2.2,0.4){\tiny$^{\otimes r-1}$}
     \put(0,-0.2){$\circ$}
     \put(1.4,-0.2){$\bullet$}
     \put(2,-0.2){$\bullet$}
     \put(3.4,-0.2){$\bullet$}
     \end{picture}}
\newsavebox{\boxpositionerdt} 
   \savebox{\boxpositionerdt}
   { \begin{picture}(2.9,1.5)
     \put(1.2,0.2){\line(0,1){0.9}}
     \put(2.9,0.2){\line(0,1){0.9}}
     \put(1.2,1.1){\line(1,0){1.7}}
     \put(0.05,-0.05){$^\uparrow$}
     \put(0.2,0.4){\tiny$^{\otimes dt}$}
     \put(1.65,-0.05){$^\uparrow$}
     \put(1.8,0.4){\tiny$^{\otimes dt}$}
     \put(0,-0.2){$\circ$}
     \put(1,-0.2){$\circ$}
     \put(1.6,-0.2){$\bullet$}
     \put(2.7,-0.2){$\bullet$}
     \end{picture}}
\newsavebox{\boxpositioners} 
   \savebox{\boxpositioners}
   { \begin{picture}(2.6,1.5)
     \put(0.9,0.2){\line(0,1){0.9}}
     \put(2.3,0.2){\line(0,1){0.9}}
     \put(0.9,1.1){\line(1,0){1.4}}
     \put(0.05,-0.05){$^\uparrow$}
     \put(0.2,0.4){\tiny$^{\otimes s}$}
     \put(1.35,-0.05){$^\uparrow$}
     \put(1.5,0.4){\tiny$^{\otimes s}$}
     \put(0,-0.2){$\circ$}
     \put(0.7,-0.2){$\circ$}
     \put(1.3,-0.2){$\bullet$}
     \put(2.1,-0.2){$\bullet$}
     \end{picture}}
\newsavebox{\boxpositionerdinv} 
   \savebox{\boxpositionerdinv}
   { \begin{picture}(2.6,1.5)
     \put(0.9,0.2){\line(0,1){0.9}}
     \put(2.3,0.2){\line(0,1){0.9}}
     \put(0.9,1.1){\line(1,0){1.4}}
     \put(0.05,-0.05){$^\uparrow$}
     \put(0.2,0.4){\tiny$^{\otimes d}$}
     \put(1.35,-0.05){$^\uparrow$}
     \put(1.5,0.4){\tiny$^{\otimes d}$}
     \put(0,-0.2){$\circ$}
     \put(0.7,-0.2){$\bullet$}
     \put(1.3,-0.2){$\bullet$}
     \put(2.1,-0.2){$\circ$}
     \end{picture}}
\newsavebox{\boxpositionersinv} 
   \savebox{\boxpositionersinv}
   { \begin{picture}(2.6,1.5)
     \put(0.9,0.2){\line(0,1){0.9}}
     \put(2.3,0.2){\line(0,1){0.9}}
     \put(0.9,1.1){\line(1,0){1.4}}
     \put(0.05,-0.05){$^\uparrow$}
     \put(0.2,0.4){\tiny$^{\otimes s}$}
     \put(1.35,-0.05){$^\uparrow$}
     \put(1.5,0.4){\tiny$^{\otimes s}$}
     \put(0,-0.2){$\circ$}
     \put(0.7,-0.2){$\bullet$}
     \put(1.3,-0.2){$\bullet$}
     \put(2.1,-0.2){$\circ$}
     \end{picture}}
\newsavebox{\boxpositionerdpluszwei}
   \savebox{\boxpositionerdpluszwei}
   { \begin{picture}(3.4,1.5)
     \put(1.6,0.2){\line(0,1){0.9}}
     \put(3.0,0.2){\line(0,1){0.9}}
     \put(1.6,1.1){\line(1,0){1.4}}
     \put(0.05,-0.05){$^\uparrow$}
     \put(0.2,0.4){\tiny$^{\otimes d+2}$}
     \put(2.05,-0.05){$^\uparrow$}
     \put(2.2,0.4){\tiny$^{\otimes d}$}
     \put(0,-0.2){$\circ$}
     \put(1.4,-0.2){$\bullet$}
     \put(2,-0.2){$\bullet$}
     \put(2.8,-0.2){$\bullet$}
     \end{picture}}
\newsavebox{\boxpositionersminuszwei}
   \savebox{\boxpositionersminuszwei}
   { \begin{picture}(3.4,1.5)
     \put(1,0.2){\line(0,1){0.9}}
     \put(3.0,0.2){\line(0,1){0.9}}
     \put(1,1.1){\line(1,0){2}}
     \put(0.05,-0.05){$^\uparrow$}
     \put(0.2,0.4){\tiny$^{\otimes s}$}
     \put(1.45,-0.05){$^\uparrow$}
     \put(1.6,0.4){\tiny$^{\otimes s-2}$}
     \put(0,-0.2){$\circ$}
     \put(0.8,-0.2){$\bullet$}
     \put(1.4,-0.2){$\bullet$}
     \put(2.8,-0.2){$\bullet$}
     \end{picture}}
\newsavebox{\boxpositionerrnull}
   \savebox{\boxpositionerrnull}
   { \begin{picture}(3.8,1.5)
     \put(1.2,0.2){\line(0,1){0.9}}
     \put(3.4,0.2){\line(0,1){0.9}}
     \put(1.2,1.1){\line(1,0){2.2}}
     \put(0.05,-0.05){$^\uparrow$}
     \put(0.2,0.4){\tiny$^{\otimes r_0}$}
     \put(1.65,-0.05){$^\uparrow$}
     \put(1.8,0.4){\tiny$^{\otimes r_0-2}$}
     \put(0,-0.2){$\circ$}
     \put(1,-0.2){$\bullet$}
     \put(1.6,-0.2){$\bullet$}
     \put(3.2,-0.2){$\bullet$}
     \end{picture}}
\newsavebox{\boxpositionerwbwb}
   \savebox{\boxpositionerwbwb}
   { \begin{picture}(1.8,1)
     \put(0.6,0.2){\line(0,1){0.6}}
     \put(1.4,0.2){\line(0,1){0.6}}
     \put(0.6,0.8){\line(1,0){0.8}}
     \put(0.05,-0.05){$^\uparrow$}
     \put(0.85,-0.05){$^\uparrow$}
     \put(0,-0.2){$\circ$}
     \put(0.4,-0.2){$\bullet$}
     \put(0.8,-0.2){$\circ$}
     \put(1.2,-0.2){$\bullet$}
     \end{picture}}
\newsavebox{\boxpositionerwwbb}
   \savebox{\boxpositionerwwbb}
   { \begin{picture}(1.8,1)
     \put(0.6,0.2){\line(0,1){0.6}}
     \put(1.4,0.2){\line(0,1){0.6}}
     \put(0.6,0.8){\line(1,0){0.8}}
     \put(0.05,-0.05){$^\uparrow$}
     \put(0.85,-0.05){$^\uparrow$}
     \put(0,-0.2){$\circ$}
     \put(0.4,-0.2){$\circ$}
     \put(0.8,-0.2){$\bullet$}
     \put(1.2,-0.2){$\bullet$}
     \end{picture}}
\newsavebox{\boxpositionerrevwbwb}
   \savebox{\boxpositionerrevwbwb}
   { \begin{picture}(1.8,1)
     \put(0.2,0.2){\line(0,1){0.6}}
     \put(1.0,0.2){\line(0,1){0.6}}
     \put(0.2,0.8){\line(1,0){0.8}}
     \put(0.45,-0.05){$^\uparrow$}
     \put(1.25,-0.05){$^\uparrow$}
     \put(0,-0.2){$\circ$}
     \put(0.4,-0.2){$\bullet$}
     \put(0.8,-0.2){$\circ$}
     \put(1.2,-0.2){$\bullet$}
     \end{picture}}
\newsavebox{\boxpositionersalphaw} 
   \savebox{\boxpositionersalphaw}
   { \begin{picture}(2.6,1.5)
     \put(0.9,0.2){\line(0,1){0.9}}
     \put(2.3,0.2){\line(0,1){0.9}}
     \put(0.9,1.1){\line(1,0){1.4}}
     \put(0.05,-0.05){$^\uparrow$}
     \put(1.35,-0.05){$^\uparrow$}
     \put(0.2,0.4){\tiny$^{\otimes s}$}
     \put(1.5,0.4){\tiny$^{\otimes \alpha}$}
     \put(0,-0.2){$\circ$}
     \put(0.7,-0.2){$\circ$}
     \put(1.3,-0.2){$\circ$}
     \put(2.1,-0.2){$\bullet$}
     \end{picture}}
\newsavebox{\boxpositionersalphab} 
   \savebox{\boxpositionersalphab}
   { \begin{picture}(2.6,1.5)
     \put(0.9,0.2){\line(0,1){0.9}}
     \put(2.3,0.2){\line(0,1){0.9}}
     \put(0.9,1.1){\line(1,0){1.4}}
     \put(0.05,-0.05){$^\uparrow$}
     \put(1.35,-0.05){$^\uparrow$}
     \put(0.2,0.4){\tiny$^{\otimes s}$}
     \put(1.5,0.4){\tiny$^{\otimes \alpha}$}
     \put(0,-0.2){$\circ$}
     \put(0.7,-0.2){$\circ$}
     \put(1.3,-0.2){$\bullet$}
     \put(2.1,-0.2){$\bullet$}
     \end{picture}}

\newsavebox{\boxAspace}
   \savebox{\boxAspace}
   { \begin{picture}(0.1,1.5)
     \end{picture}}
\newsavebox{\boxBspace}
   \savebox{\boxBspace}
   { \begin{picture}(0,1.85)
     \end{picture}}

\newcommand{\idpartww}{\usebox{\boxidpartww}}
\newcommand{\idpartwb}{\usebox{\boxidpartwb}}
\newcommand{\idpartbw}{\usebox{\boxidpartbw}}
\newcommand{\idpartbb}{\usebox{\boxidpartbb}}

\newcommand{\paarpartww}{\usebox{\boxpaarpartww}}
\newcommand{\paarpartbw}{\usebox{\boxpaarpartbw}}
\newcommand{\paarpartwb}{\usebox{\boxpaarpartwb}}
\newcommand{\paarpartbb}{\usebox{\boxpaarpartbb}}

\newcommand{\baarpartbw}{\usebox{\boxbaarpartbw}}
\newcommand{\baarpartwb}{\usebox{\boxbaarpartwb}}

\newcommand{\singletonw}{\usebox{\boxsingletonw}}
\newcommand{\singletonb}{\usebox{\boxsingletonb}}
\newcommand{\downsingletonw}{\usebox{\boxdownsingletonw}}

\newcommand{\vierpartwbwb}{\usebox{\boxvierpartwbwb}}

\newcommand{\vierpartwwbb}{\usebox{\boxvierpartwwbb}}

\newcommand{\vierpartrotwbwb}{\usebox{\boxvierpartrotwbwb}}

\newcommand{\vierpartrotbwbw}{\usebox{\boxvierpartrotbwbw}}
\newcommand{\vierpartrotwwww}{\usebox{\boxvierpartrotwwww}}

\newcommand{\crosspartwbbw}{\usebox{\boxcrosspartwbbw}}
\newcommand{\crosspartbwwb}{\usebox{\boxcrosspartbwwb}}
\newcommand{\crosspartwwww}{\usebox{\boxcrosspartwwww}}
\newcommand{\crosspartbbbb}{\usebox{\boxcrosspartbbbb}}
\newcommand{\halflibpartwwwwww}{\usebox{\boxhalflibpartwwwwww}}
\newcommand{\positionerd}{\usebox{\boxpositionerd}}
\newcommand{\positionerrpluseins}{\usebox{\boxpositionerrpluseins}}

\newcommand{\positionerwbwb}{\usebox{\boxpositionerwbwb}}
\newcommand{\positionerwwbb}{\usebox{\boxpositionerwwbb}}

\DeclareMathOperator{\rot}{rot}
\DeclareMathOperator{\nest}{nest}
\DeclareMathOperator{\Mor}{Mor}
\DeclareMathOperator{\id}{id}

\DeclareMathOperator{\spanlin}{span}

\newcommand{\subsetup}{\begin{rotate}{90}$\subset$\end{rotate}}

\newcommand{\categ}[3]{{#1}_{#2}(#3)}
\newcommand{\categg}[2]{{#1}_{#2}}
\newcommand{\glob}{\textnormal{glob}}
\newcommand{\loc}{\textnormal{loc}}
\newcommand{\grp}{\textnormal{grp}}
\newcommand{\freeglued}{\tilde *}
\newcommand{\tensorglued}{\tilde \times}

\begin{document}
\title[Unitary Easy Quantum Groups]{Unitary Easy Quantum Groups:\\
the free case and the group case}
\author{Pierre Tarrago and Moritz Weber}
\address{Universit\'e de Tours, Laboratoire de Math\'ematiques et Physique Th\'eorique, UFR Sciences et Techniques, Parc de Grandmont, 37200 Tours, France}
\email{Pierre.Tarrago@lmpt.univ-tours.fr}
\address{Saarland University, Fachbereich Mathematik, Postfach 151150,
66041 Saarbr\"ucken, Germany}
\email{weber@math.uni-sb.de}
\date{\today}

\begin{abstract}
Easy quantum groups have been studied intensively since the time they were introduced by Banica and Speicher in 2009. They arise as a subclass of ($C^*$-algebraic) compact matrix quantum groups in the sense of Woronowicz. Due to some Tannaka-Krein type result, they are completely determined by the combinatorics of categories of (set theoretical) partitions.
So far, only orthogonal easy quantum groups have been considered in order to understand quantum subgroups of the free orthogonal quantum group $O_n^+$.

 We now give a definition of unitary easy quantum groups using colored partitions to tackle the problem of finding quantum subgroups of $U_n^+$. In the free case (i.e. restricting to noncrossing partitions), the corresponding categories of partitions have recently been classified by the authors by purely combinatorial means. There are ten series showing up each indexed by one or two discrete parameters, plus two additional quantum groups. We now present the quantum group picture of it and investigate them in detail. We show how they can be constructed from other known examples using generalizations of Banica's free complexification. For doing so, we introduce new kinds of products between quantum groups.
 
 We also study the notion of easy groups.

\end{abstract}

\maketitle
\section*{Introduction}

In order to provide a new notion of symmetries adapted to the situation in operator algebras, Woronowicz introduced compact matrix quantum groups in 1987 in \cite{woronowicz1987compact}. The idea is roughly to take a compact Lie group $G\subset M_n(\C)$ and to pass to the algebra $C(G)$ of continuous functions over it. It turns out that this algebra fulfills all axioms of a $C^*$-algebra with the special feature that the multiplication is commutative. If we now also dualize main properties of the group law $\mu:G\times G\to G$ to properties of a map $\Delta:C(G)\to C(G\times G)\cong C(G)\otimes C(G)$ (the comultiplication), and consider noncommutative $C^*$-algebras  equiped with such a map $\Delta$, we are right in the heart of the definition of a compact matrix quantum group (see Section \ref{SectCMQG} for details). It generalizes the notion of a compact matrix group. The reader not familiar with quantum groups might also first jump to Section \ref{SectUnitaryGroups} and read it as a motivation.

By some Tannaka-Krein type result of Woronowicz  \cite{woronowicz1988tannaka} compact matrix quantum groups are governed by their intertwiner spaces which form a tensor category. Banica and Speicher observed in \cite{banica2009liberation} that one can define operations on set theoretical partitions which translate one-to-one to operations on intertwiners of certain quantum groups. This lead to the definition of easy quantum groups in 2009. They all sit in between the symmetric group $S_n$ and the free orthogonal quantum group $O_n^+$ constructed by Wang in \cite{wang1995free}. We now extend this approach to quantum groups $S_n\subset G\subset U_n^+$, where $U_n^+$ is Wang's free unitary quantum group \cite{wang1995free}. This approach will help to understand $U_n^+$ better by analyzing its quantum  subgroups.

Let us briefly sketch the main ideas and the main result of this article. Our basic objects are two-colored set theoretical partitions, i.e. we study decompositions of the set of $k+l$ points into disjoint subsets (the blocks), for $k,l\in\N_0$. In addition, these points may be colored either white or black. We represent such a partition pictorially by connecting $k$ upper points with $l$ lower points using strings according to the block pattern. Given two such partitions, we may form the tensor product (placing them side by side), the composition (placing one above the other), and the involution (reflecting a partition at the horizontal axis). If a set of partitions is closed under these operations and if it contains certain base partitions, it is called a category of partitions (see Section \ref{SectCateg}). Now, any category of partitions gives rise to a compact matrix quantum group $G$ with $S_n\subset G\subset U_n^+$ via a functor $p\mapsto T_p$ assigning linear maps (i.e. intertwiners) to partitions $p$. Following Banica and Speicher's definition in the orthogonal case (partitions without colors), we say that a quantum group $G$ is a unitary easy quantum group (see Section \ref{SectUnitaryEasy}), if its intertwiner spaces arise in this way from a category of partitions. Here, the color pattern of a partition rules to which tensor product combination of the fundamental representation $u$ and its contragredient $\bar u$ the intertwiner $T_p$ is applied. In the orthogonal case, we have $u=\bar u$ and no colors are needed on the level of partitions. In this sense, the coloring of the partitions is due to the fact that the generators $u_{ij}$ of our quantum groups are no longer self-adjoint when passing from the orthogonal to the unitary situation.

This is a way of producing examples of quantum groups having a purely combinatorial structure underlying. It is natural to ask how rich this machinery is, and we therefore wish to classify all categories of partitions -- hence, all unitary easy quantum groups. Note that in the orthogonal case (no colors) this has recently been achieved \cite{raum2013full}. In the unitary case (two colors) however, little is known. In a recent article  \cite{tarragowebercombina}, the authors classified all categories of (two-colored) noncrossing partitions. This is called the free case, since the corresponding quantum groups all sit in between Wang's free symmetric quantum group $S_n^+$ and $U_n^+$. While there are only seven free orthogonal easy quantum groups, the world of free unitary easy quantum groups is way richer. In the present article, we introduce the unitary easy quantum groups associated to the categories of partitions found in \cite{tarragowebercombina}. There are ten series of quantum groups each indexed by one or two parameters from the natural numbers, and two additional quantum groups. They all can be constructed from some basic quantum groups (including the free easy orthogonal ones) using either free or tensor $d$-complexifications. We define these constructions as a generalization of Banica's free complexification \cite{banica2008note} by replacing $\Z$ by $\Z/d\Z$. Note that from the classification in \cite{tarragowebercombina}, we know that the categories of noncrossing partitions fall into the class of globally colorized ones and the locally colorized ones. The first case is much easier (for instance the $u_{ij}$ are normal) and the quantum groups can be written as tensor $d$-complexifications. In the latter case, we have to involve also free $d$-complexifications.

As a main result of this article, we prove that all free easy quantum groups may be constructed from a few basic quantum groups using these complexifications, see Theorems \ref{ThmEasyGlob} and \ref{ThmEasyLoca}.
We compare our easy quantum groups with other known examples and we infer that for instance Banica and Vergnioux's quantum reflection groups $H_n^{s+}$ are easy. We also treat the case of unitary easy groups, see Section \ref{SectUnitaryGroups}. We end this article with many remarks on the use of easy quantum groups and on open problems, see Section \ref{SectConcluding}.

\section*{Acknowledgements}

We thank Teo Banica, Stephen Curran and Roland Speicher for sending us an unpublished draft \cite{speicherunpublished} of their work on the definition and classification of unitary easy quantum groups. Some parts of this article may be found in their draft, too. We thank Adam Skalski for discussions on different versions of quantum groups (full, reduced etc.).

Tarrago was supported by the Universit\'e Franco-Allemande. Both authors were partially funded by the ERC Advanced Grant on Non-Commutative Distributions in Free Probability, held by Roland Speicher.

\section{Categories of two-colored partitions}

Let us first briefly introduce the combinatorial objects underlying the unitary easy quantum groups. See \cite{tarragowebercombina} for more details concerning this section. The work in \cite{tarragowebercombina} is an extension of the well-known combinatorics of orthogonal easy quantum groups which were introduced by Banica and Speicher in \cite{banica2009liberation}.

\subsection{Partitions}

For $k,l\in\N_0:=\{0,1,2,\ldots\}$, we consider the finite set given by $k$ ``upper'' points and $l$ ``lower'' points. These points are colored either in white or in black and we say that these colors are \emph{inverse} to each other. A \emph{(two-colored) partition} is a  decomposition of this set into disjoint subsets, the \emph{blocks}. We usually draw these partitions by placing the $k$ ``upper'' points on a top row and the $l$ ``lower'' points on a bottom row, connecting them by strings according to the block pattern. Here are two examples of such partitions. The first one consists of three blocks whereas the second one has four blocks (one of these four blocks is a singleton).

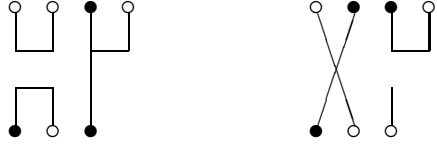
\setlength{\unitlength}{0.5cm}
\begin{center}
\begin{picture}(12,4)
\put(-1,4.35){\partii{1}{1}{2}}
\put(-1,4.35){\partii{1}{3}{4}}
\put(-1,0.35){\uppartii{1}{1}{2}}
\put(-1,0.35){\upparti{2}{3}}
\put(0.05,0){$\bullet$}
\put(1.05,0){$\circ$}
\put(2.05,0){$\bullet$}
\put(0.05,3.3){$\circ$}
\put(1.05,3.3){$\circ$}
\put(2.05,3.3){$\bullet$}
\put(3.05,3.3){$\circ$}
\put(8.3,3.35){\line(1,-3){1}}
\put(7,4.35){\partii{1}{3}{4}}
\put(8.3,0.35){\line(1,3){1}}
\put(7,0.35){\upparti{1}{3}}
\put(8.05,0){$\bullet$}
\put(9.05,0){$\circ$}
\put(10.05,0){$\circ$}
\put(8.05,3.3){$\circ$}
\put(9.05,3.3){$\bullet$}
\put(10.05,3.3){$\bullet$}
\put(11.05,3.3){$\circ$}
\end{picture}
\end{center}

We denote by $P^{\twocol}(k,l)$ the set of all such partitions, and by $P^{\twocol}$  the union of all $P^{\twocol}(k,l)$, for $k,l\in\N_0$. By  $NC^{\twocol}(k,l)$  and $NC^{\twocol}$ respectively, we denote the set of \emph{noncrossing partitions}, i.e. of partitions whose lines can be drawn in such a way that they do not cross. In the above examples, the first partition is in $NC^{\twocol}$ whereas the second is not.

\subsection{Operations on partitions}\label{SectOperations}

Let us now turn to operations on the set $P^{\twocol}$.
\begin{itemize}
 \item  The \emph{tensor product} of two partitions $p\in P^{\twocol}(k,l)$ and $q\in P^{\twocol}(k',l')$ is the partition $p\otimes q\in P^{\twocol}(k+k',l+l')$ obtained by horizontal concatenation (writing $p$ and $q$ side by side).
 \item The \emph{composition} of two partitions $q\in P^{\twocol}(k,l)$ and $p\in P^{\twocol}(l,m)$ is the partition  $pq\in P^{\twocol}(k,m)$ obtained by vertical concatenation (writing $p$ below $q$ and removing the $l$ middle points).  Note that we can compose these two partitions  only if
the colorings match, i.e. the color of the $j$-th lower point of $q$ coincides with the color of the $j$-th upper point of $p$, for all $1\leq j\leq l$.
 \item The \emph{involution} of a partition $p\in P^{\twocol}(k,l)$ is given by the reflection  $p^*\in P^{\twocol}(l,k)$ at the horizontal axis.
 \item The \emph{verticolor reflection} of a partition $p\in P^{\twocol}(k,l)$ is given by the partition $\tilde p\in P^{\twocol}(k,l)$ obtained from reflecting $p$ at the vertical axis and inverting all colors of the points.
 \item We may also produce \emph{rotated versions} of a partition  by shifting the leftmost upper point to the  lower line or vice versa; likewise on the right hand side. A shifted point still belongs to the same block as before, but its color is inverted when being moved from the upper line to the lower line or vice versa.
\end{itemize}

See \cite{tarragowebercombina} for examples of these operations.

\subsection{Categories of partitions} \label{SectCateg}

A collection $\CC\subset P^{\twocol}$ of subsets $\CC(k,l)\subseteq P^{\twocol}(k,l)$ (for all $k,l\in\N_0$) is a \emph{category of partitions}, if it is closed under the tensor product, the composition and the involution, and if it contains the \emph{pair partitions} $\paarpartwb\in P^{\twocol}(0,2)$ and  $\paarpartbw\in P^{\twocol}(0,2)$ as well as the \emph{identity partitions} $\idpartww\in P^{\twocol}(1,1)$ and $\idpartbb\in P^{\twocol}(1,1)$. It can be shown that categories of partitions are closed under rotation and verticolor reflection  \cite[Lemma 1.1]{tarragowebercombina}. We write $\CC=\langle p_1,\ldots, p_n\rangle$ if $\CC$ is the smallest category containing the partitions $p_1,\ldots, p_n\in P^{\twocol}$. We then say that $\CC$ is \emph{generated} by $p_1,\ldots, p_n$. We usually omit to write down the generators $\paarpartwb$ and $\idpartww$, since they are always in a category, by definition.

\subsection{The non-colored case}\label{SectOnecolored}

When defining orthogonal easy quantum groups, Banica and Speicher \cite{banica2009liberation} introduced categories of partitions whose points are not colored. These categories are in one-to-one correspondence with categories of two-colored partitions containing the partition $\paarpartww$ (or equivalently, by verticolor reflection, $\paarpartbb$). Note that any such category contains $\idpartwb$ and $\idpartbw$ by rotation. Composition with these partitions yields that we can change the colors of the points arbitrarily -- in this sense, they appear to be non-colored. More precisely, let $\Psi:P^{\twocol}\to P$ be the map given by forgetting the colors of a two-colored partition. Here $P$ denotes the set of all non-colored partitions. We have  \cite[Prop. 1.4]{tarragowebercombina}: If $\CC\subset P$ is a category of non-colored partitions, then its preimage $\Psi^{-1}(\CC)\subset P^{\twocol}$ is a category of two-colored partitions containing  $\paarpartww$. Conversely, if  $\CC\subset P^{\twocol}$ is a category of two-colored partitions containing the unicolored pair partition $\paarpartww$, then $\Psi(\CC)\subset P$ is a category of non-colored partitions (we need $\paarpartww$ to show that it is closed under composition) and $\Psi^{-1}(\Psi(\CC))=\CC$.

\subsection{The cases of categories of non-crossing partitions}\label{SectCases}

The classification of categories of noncrossing partitions in \cite{tarragowebercombina} is based on a division of the categories into four cases. By $\singletonw\in P^{\twocol}(0,1)$ we denote the \emph{singleton partition} consisting of a single, white point; likewise $\singletonb$ if this point is black.

\begin{defn}\label{DefCases}
Let $\CC\subset P^{\twocol}$ be a category of partitions. We say that:
\begin{itemize}
\item $\CC$ is in \emph{case $\OOO$}, if $\singletonw\otimes\singletonb\notin\CC$ and $\vierpartwbwb\notin\CC$.
\item $\CC$ is in \emph{case $\BBB$}, if $\singletonw\otimes\singletonb\in\CC$ and $\vierpartwbwb\notin\CC$.
\item $\CC$ is in \emph{case $\HHH$}, if $\singletonw\otimes\singletonb\notin\CC$ and $\vierpartwbwb\in\CC$.
\item $\CC$ is in \emph{case $\SSS$}, if $\singletonw\otimes\singletonb\in\CC$ and $\vierpartwbwb\in\CC$.
\end{itemize}
\end{defn}

These four case behave quite differently, as can be seen in the next lemma.

\begin{lem}[{\cite[Lemmas 1.3 and 2.1]{tarragowebercombina}}]
Let $\CC\subset P^{\twocol}$ be a category of partitions.
\begin{itemize}
\item[(a)] If $\singletonw\otimes\singletonb\in\CC$, then $\CC$ is closed under permutation of colors of neighbouring singletons both either sitting on the upper row, or both on the lower one. Furthermore, we may disconnect any point from a block and turn it into a singleton.
\item[(b)] If $\singletonw\otimes\singletonb\notin\CC$, then all blocks of  partitions $p\in\CC$ have length at least two.
\item[(c)] If $\vierpartwbwb\in\CC$, we may connect neighbouring blocks if they meet at two points (on the same row) with inverse colors. If we even have $\vierpartwwbb\in\CC$, then $\CC$ is closed under permutation of colors of neighbouring points (on the same row) belonging to the same block; furthermore, we may connect neighbouring blocks independent of the colors of the points at which they meet.
\item[(d)] If $\vierpartwbwb\notin\CC$, then all blocks of partitions $p\in\CC$ have length at most two.
\end{itemize}
\end{lem}

Each of these four cases is subdivided again into two cases.

\begin{defn}\label{DefGlobalColor}
A category of partitions $\CC\subset P^{\twocol}$ is 
\begin{itemize}
 \item \emph{globally colorized}, if $\paarpartww\otimes\paarpartbb\in\CC$
 \item and \emph{locally colorized} if $\paarpartww\otimes\paarpartbb\notin\CC$.
\end{itemize}
\end{defn}

Globally colorized categories of partitions are closed under under permutation of colors of the points (compose with $\idpartwb\otimes\idpartbw$, see \cite[Lemma 1.3]{tarragowebercombina}). Hence, the colorization of partitions only matters on a global level, i.e. the difference $c(p)$ between the number of white and black points of a partition $p\in\CC$ yields a crucial parameter, the \emph{global parameter} $k(\CC)$ given by the minimum of all $c(p)>0$, $p\in\CC$, see \cite[Def. 2.5]{tarragowebercombina}. Locally colorized categories in turn are more sensitive for the specific color patterns and we need additional parameters, mainly the \emph{local parameter} $d(\CC)$ which is based on the numbers $c(p_1)$ for subpartitions $p_1$ of $p\in\CC$ that sit between two legs of $p$. We distinguish between the cases when the two points of these legs are colored by the same color or not.

\subsection{All categories of (two-colored) noncrossing partitions}\label{SectClassif}

We review the classification of all categories of noncrossing  partitions obtained in \cite{tarragowebercombina}. By $b_k\in P^{\twocol}(0,k)$ we denote the partition consisting of a single block of $k$ points all of which are white, while $\tilde b_k\in P^{\twocol}(0,k)$ denotes its verticolor reflection (see Section \ref{SectOperations}), i.e. all $k$ points are black.

\begin{thm}[{\cite[Thm. 7.1]{tarragowebercombina}}]\label{ThmClassiGlob}
Let $\CC\subset NC^{\twocol}$ be a \emph{globally colorized} category of noncrossing partitions. Then it coincides with one of the following categories.
\begin{itemize}
 \item[Case $\OOO$:] $\categ{\OOO}{\glob}{k}=\langle\paarpartww^{\otimes \frac{k}{2}},\paarpartww\otimes\paarpartbb\rangle$ for $k\in 2\N_0$
 \item[Case $\HHH$:] $\categ{\HHH}{\glob}{k}=\langle b_k,\vierpartwbwb,\paarpartww\otimes\paarpartbb\rangle$ for $k\in 2\N_0$
 \item[Case $\SSS$:] $\categ{\SSS}{\glob}{k}=\langle \singletonw^{\otimes k},\vierpartwbwb,\singletonw\otimes\singletonb,\paarpartww\otimes\paarpartbb\rangle$ for $k\in \N_0$
 \item[Case $\BBB$:] $\categ{\BBB}{\glob}{k}=\langle \singletonw^{\otimes k}, \singletonw\otimes\singletonb,\paarpartww\otimes\paarpartbb\rangle$ for $ k\in 2\N_0$ 
 \item[or] $\categ{\BBB'}{\glob}{k}=\langle \singletonw^{\otimes k},\positionerwwbb,\singletonw\otimes\singletonb,\paarpartww\otimes\paarpartbb\rangle$ for $k\in \N_0$
\end{itemize}
\end{thm}

The parameters $k\in\N_0$ in the above theorem are derived from the global parameters $k(\CC)$ as mentioned in the previous subsection. In the locally colorized case, we also need to take into account the local parameter $d(\CC)$  which is always a divisor of $k(\CC)$. By $d\vert k$ we mean that $k$ is a multiple of $d$. In particular, $d=0$ implies $k=0$, but the case $d\neq 0$ with $k=0$ may occur.

\begin{thm}[{\cite[Thm. 7.2]{tarragowebercombina}}]\label{ThmClassiLoc}
Let $\CC\subset NC^{\twocol}$ be a \emph{locally colorized} category of noncrossing partitions.Then it coincides with one of the following categories.
\begin{itemize}
 \item[Case $\OOO$:] $\categg{\OOO}{\loc}=\langle\emptyset\rangle$
 \item[Case $\HHH$:] $\categg{\HHH'}{\loc}=\langle\vierpartwbwb\rangle$ 
 \item[or] $\categ{\HHH}{\loc}{k,d}=\langle b_k,b_d\otimes\tilde b_d,\vierpartwwbb,\vierpartwbwb\rangle$ for $k,d\in\N_0\backslash\{1,2\}$, $d\vert k$ 
 \item[Case $\SSS$:] $\categ{\SSS}{\loc}{k,d}=\langle \singletonw^{\otimes k},\positionerd,\vierpartwbwb,\singletonw\otimes\singletonb\rangle$ for $k,d\in\N_0\backslash\{1\}$, $d\vert k$
 \item[Case $\BBB$:] $\categ{\BBB}{\loc}{k,d}=\langle\singletonw^{\otimes k}, \positionerd,\singletonw\otimes\singletonb\rangle$ for $k,d\in \N_0$, $d\vert k$
 \item[or] $\categ{\BBB'}{\loc}{k,d,0}=\langle\singletonw^{\otimes k}, \positionerd,\positionerwbwb,\singletonw\otimes\singletonb\rangle$ for $k,d\in \N_0\backslash\{1\}$, $d\vert k$
 \item[or] $\categ{\BBB'}{\loc}{k,d,\frac{d}{2}}=\langle\singletonw^{\otimes k}, \positionerd,\positionerrpluseins,\singletonw\otimes\singletonb\rangle$ for $k\in \N_0\backslash\{1\}$, $d\in 2\N_0\backslash\{0,2\}$, $d\vert k$ and $r=\frac{d}{2}$.
\end{itemize}
\end{thm}

\section{$C^*$-algebraic relations associated to partitions}\label{SectCStar}

We can associate $C^*$-algebras to categories of partitions by  associating relations to partitions. This is the main step in the direction of defining unitary easy quantum groups.

\begin{defn}\label{DefDelta}
Let $p\in P^{\twocol}(k,l)$ and let $\alpha=(\alpha_1,\ldots,\alpha_k)$ and $\beta=(\beta_1,\ldots,\beta_l)$ be multi indices with $\alpha_i, \beta_j\in\{1,\ldots,n\}$. We decorate the upper points of $p$ with $\alpha$ and the lower ones with $\beta$. If now for every block of $p$ all of the corresponding indices coincide, we put $\delta_p(\alpha,\beta):=1$; otherwise $\delta_p(\alpha,\beta):=0$. Here, the colorization of $p$ is irrelevant.
\end{defn}

\begin{ex}
The following choices of $\alpha$ and $\beta$ yield:
\setlength{\unitlength}{0.5cm}
\begin{center}
\begin{picture}(20,5)
\put(-1,4.85){\partii{1}{1}{2}}
\put(-1,4.85){\partii{1}{3}{4}}
\put(-1,0.85){\uppartii{1}{1}{2}}
\put(-1,0.85){\upparti{2}{3}}
\put(0.05,4){1}
\put(1.05,4){1}
\put(2.05,4){5}
\put(3.05,4){5}
\put(0.05,0.1){2}
\put(1.05,0.1){2}
\put(2.05,0.1){5}
\put(5,2){$\delta=1$}
\put(10,4.85){\partii{1}{1}{2}}
\put(10,4.85){\partii{1}{3}{4}}
\put(10,0.85){\uppartii{1}{1}{2}}
\put(10,0.85){\upparti{2}{3}}
\put(11.05,4){1}
\put(12.05,4){1}
\put(13.05,4){5}
\put(14.05,4){1}
\put(11.05,0.1){2}
\put(12.05,0.1){2}
\put(13.05,0.1){5}
\put(16,2){$\delta=0$}
\end{picture}
\end{center}
\end{ex}

\begin{defn}\label{DefRel}
Let $n\in\N$ and let $A$ be a $C^*$-algebra generated by $n^2$ elements $u_{ij}$, $1\leq i,j\leq n$.  Let $p\in P^{\twocol}(k,l)$ be a partition and let $r=(r_1,\ldots,r_k)\in\{\circ,\bullet\}^k$ be its upper color pattern and $s=(s_1,\ldots,s_l)\in\{\circ,\bullet\}^l$ be its lower color pattern. We put $u_{ij}^\circ:=u_{ij}$ and $u_{ij}^\bullet:=u_{ij}^*$. 

We say that the generators $u_{ij}$ \emph{fulfill the relations} $R(p)$, if for all $\beta_1,\ldots, \beta_l\in\{1,\ldots,n\}$ and for all $i_1,\ldots,i_k\in\{1,\ldots,n\}$, we have:
\[\sum_{\alpha_1,\ldots,\alpha_k=1}^n \delta_p(\alpha,\beta) u^{r_1}_{\alpha_1i_1}\ldots u^{r_k}_{\alpha_ki_k}=\sum_{\gamma_1,\ldots,\gamma_l=1}^n \delta_p(i,\gamma) u^{s_1}_{\beta_1\gamma_1}\ldots u^{s_l}_{\beta_l\gamma_l}\]

The left-hand side of the equation is $\delta_p(\emptyset,\beta)$ if $k=0$ and analogous $\delta_p(i,\emptyset)$ for the right-hand side in the case $l=0$.
\end{defn}

Using this definition, we can now give a list of relations associated to partitions that appeared throughout the classification (see Section \ref{SectClassif}) of categories of noncrossing partitions. We denote by $u$ the matrix $u=(u_{ij})_{1\leq i,j\leq n}$, and $\bar u:=(u_{ij}^*)$. Furthermore, $\rot_t(p)\in P^{\twocol}(t,k)$ denotes the partitions obtained from $p\in P^{\twocol}(0,k+t)$ by rotating the last $t$ points to the upper line. If we simply write $\rot(p)$, we do not specify which or how many of the points are rotated. It is often more convenient to consider the relations of a partition in some rotated form rather than of the partition itself. 

Regarding the relations, it is reasonable to reformulate the category $\OOO_{\glob}(k)$ using the partition $\paarpartww^{\nest(k)}$  obtained from nesting the partition $\paarpartww$ $k$-times into itself, i.e.:

\setlength{\unitlength}{0.5cm}
\begin{center}
\begin{picture}(10,4)
\put(0,1){$\paarpartww^{\nest(3)}=$}
\put(4,0.35){\uppartii{3}{1}{6}}
\put(4,0.35){\uppartii{2}{2}{5}}
\put(4,0.35){\uppartii{1}{3}{4}}
\put(5.05,0){$\circ$}
\put(6.05,0){$\circ$}
\put(7.05,0){$\circ$}
\put(8.05,0){$\circ$}
\put(9.05,0){$\circ$}
\put(10.05,0){$\circ$}
\end{picture}
\end{center}

\begin{lem}
If $\CC\subset P^{\twocol}$ is a category of partitions, then $\paarpartww^{\otimes k}\in\CC$ if and only if $\paarpartww^{\nest(k)}\in\CC$. In particular $\OOO_{\glob}(k)=\langle\paarpartww^{\nest(\frac{k}{2})},\paarpartww\otimes\paarpartbb\rangle$.
\end{lem}
\begin{proof}
Let $k\in\N_0\backslash\{0\}$ and $\paarpartww^{\otimes k}\in\CC$. We have $\paarpartwb^{\nest(k)}\in\CC$ by \cite[Lemma 1.1(d)]{tarragowebercombina} and hence $\paarpartww^{\otimes k}\otimes\paarpartwb^{\nest(k)}\in\CC$. By \cite[Lemma 1.1]{tarragowebercombina}, we know that $\paarpartww\otimes\paarpartbb\in\CC$ and hence $\CC$ is closed under permutation of colors \cite[Lemma 1.3]{tarragowebercombina}. Thus $\paarpartwb^{\otimes k}\otimes\paarpartww^{\nest(k)}\in\CC$ and therefore $\paarpartww^{\nest(k)}\in\CC$ using \cite[Lemma 1.1(b)]{tarragowebercombina}. The converse direction is similar.
\end{proof}

The next relations can  directly be derived from Definition \ref{DefRel}.

\[
R(\paarpartwb): \sum_k u_{ik}u_{jk}^*=\delta_{ij}, \textnormal{ i.e. } uu^*=1 
\qquad\qquad\qquad\qquad\qquad\qquad\qquad\qquad\qquad\qquad\qquad\qquad\qquad\qquad\qquad
\]
\[
R(\baarpartbw): \sum_k u_{ki}^*u_{kj}=\delta_{ij}, \textnormal{ i.e. } u^*u=1
\qquad\qquad\qquad\qquad\qquad\qquad\qquad\qquad\qquad\qquad\qquad\qquad\qquad\qquad\qquad
\]
\[
R(\paarpartbw): \sum_k u_{ik}^*u_{jk}=\delta_{ij}, \textnormal{ i.e. } \bar u(\bar u)^*=1
\qquad\qquad\qquad\qquad\qquad\qquad\qquad\qquad\qquad\qquad\qquad\qquad\qquad\qquad\qquad
\]
\[
R(\baarpartwb): \sum_k u_{ki}u_{kj}^*=\delta_{ij}, \textnormal{ i.e. } (\bar u)^*\bar u=1 
\qquad\qquad\qquad\qquad\qquad\qquad\qquad\qquad\qquad\qquad\qquad\qquad\qquad\qquad\qquad
\]
\[
R(\crosspartwwww)=R(\crosspartbbbb): u_{ij}u_{kl}=u_{kl}u_{ij}
\qquad\qquad\qquad\qquad\qquad\qquad\qquad\qquad\qquad\qquad\qquad\qquad\qquad\qquad\qquad
\]
\[
R(\crosspartwbbw)=R(\crosspartbwwb): u_{ij}u_{kl}^*=u_{kl}^*u_{ij}
\qquad\qquad\qquad\qquad\qquad\qquad\qquad\qquad\qquad\qquad\qquad\qquad\qquad\qquad\qquad
\]
\[
R(\idpartwb)=R(\idpartbw)=R(\rot_1(\paarpartww)): u_{ij}=u_{ij}^*, \textnormal{ i.e. } u=\bar u
\qquad\qquad\qquad\qquad\qquad\qquad\qquad\qquad\qquad\qquad\qquad\qquad\qquad\qquad\qquad
\]
\[
R(\rot_l(\paarpartww^{\nest(l)})): u_{i_1j_1}\ldots u_{i_lj_l}= u_{i_1j_1}^*\ldots u_{i_lj_l}^*
\qquad\qquad\qquad\qquad\qquad\qquad\qquad\qquad\qquad\qquad\qquad\qquad\qquad\qquad\qquad
\]
\[
R(\idpartbw\otimes \idpartwb)=R(\rot(\paarpartww\otimes\paarpartbb)): u_{ij}^*u_{kl}=u_{ij}u_{kl}^*
\qquad\qquad\qquad\qquad\qquad\qquad\qquad\qquad\qquad\qquad\qquad\qquad\qquad\qquad\qquad
\]
\[
R(\rot_1(\singletonw\otimes\singletonb)): \left(\sum_{k} u_{kj}\right)=\left(\sum_{l} u_{il}\right)
\qquad\qquad\qquad\qquad\qquad\qquad\qquad\qquad\qquad\qquad\qquad\qquad\qquad\qquad\qquad
\]
\[
R(\rot_t(\singletonw^{\otimes s+t})): \left(\sum_{k_1} u_{k_1j_1}\right)\ldots\left(\sum_{k_s} u_{k_sj_s}\right)=\left(\sum_{l_1} u_{i_1l_1}^*\right)\ldots\left(\sum_{l_t} u_{i_tl_t}^*\right)
\qquad\qquad\qquad\qquad\qquad\qquad\qquad\qquad\qquad\qquad\qquad\qquad\qquad\qquad\qquad
\]
\[
R(\singletonw^{\otimes k}):\left(\sum_{l_1} u_{l_1j_1}\right)\ldots \left(\sum_{l_k} u_{l_kj_k}\right)=1
\qquad\qquad\qquad\qquad\qquad\qquad\qquad\qquad\qquad\qquad\qquad\qquad\qquad\qquad\qquad
\]
\[
R(\vierpartrotwbwb)=R(\rot_2(\vierpartwbwb)): u_{ki}u_{kj}^*=u_{ik}u_{jk}^*=0 \emph{ if }i\neq j
\qquad\qquad\qquad\qquad\qquad\qquad\qquad\qquad\qquad\qquad\qquad\qquad\qquad\qquad\qquad
\]
\[
R(\vierpartrotbwbw): u_{ki}^*u_{kj}=u_{ik}^*u_{jk}=0 \emph{ if }i\neq j
\qquad\qquad\qquad\qquad\qquad\qquad\qquad\qquad\qquad\qquad\qquad\qquad\qquad\qquad\qquad
\]
\[
R(\vierpartrotwwww)=R(\rot_2(\vierpartwwbb)): u_{ki}u_{kj}=u_{ik}u_{jk}=0 \emph{ if }i\neq j
\qquad\qquad\qquad\qquad\qquad\qquad\qquad\qquad\qquad\qquad\qquad\qquad\qquad\qquad\qquad
\]
\[
R (\rot_d(b_d\otimes \tilde b_d)): \sum_k \delta_{i_1i_2}\delta_{i_2i_3}\ldots\delta_{i_{d-1}i_d}u_{kj_1}\ldots u_{kj_d}=\sum_l \delta_{j_1j_2}\delta_{j_2j_3}\ldots\delta_{j_{d-1}j_d}u_{i_1l}\ldots u_{i_dl}
\qquad\qquad\qquad\qquad\qquad\qquad\qquad\qquad\qquad\qquad\qquad\qquad\qquad\qquad\qquad
\]
\[
R (b_k): \sum_l u_{i_1l}\ldots u_{i_kl}=\delta_{i_1i_2}\delta_{i_2i_3}\ldots\delta_{i_{k-1}i_k}
\qquad\qquad\qquad\qquad\qquad\qquad\qquad\qquad\qquad\qquad\qquad\qquad\qquad\qquad\qquad
\]
\[
R (\rot_t(b_{s+t})): \delta_{i_1i_2}\delta_{i_2i_3}\ldots\delta_{i_{t-1}i_t}u_{i_1j_1}\ldots u_{i_1j_s}= \delta_{j_1j_2}\delta_{j_2j_3}\ldots\delta_{j_{s-1}j_s}u_{i_1j_1}^*\ldots u_{i_tj_1}^*
\qquad\qquad\qquad\qquad\qquad\qquad\qquad\qquad\qquad\qquad\qquad\qquad\qquad\qquad\qquad
\]
\[
R(\rot_{d+1}(\positionerd)):u_{ij}\left(\sum_{k_1} u_{k_1j_1}\right)\ldots \left(\sum_{k_d} u_{k_dj_d}\right)=\left(\sum_{l_1} u_{i_1l_1}\right)\ldots \left(\sum_{l_d} u_{i_dl_d}\right)u_{ij}
\qquad\qquad\qquad\qquad\qquad\qquad\qquad\qquad\qquad\qquad\qquad\qquad\qquad\qquad\qquad
\]
\[
R(\rot_2(\positionerwbwb)):u_{ij}\left(\sum_{k_1} u_{k_1j_1}^*\right)=\left(\sum_{l_1} u_{i_1l_1}\right)u_{ij}^*
\qquad\qquad\qquad\qquad\qquad\qquad\qquad\qquad\qquad\qquad\qquad\qquad\qquad\qquad\qquad
\]
\[
R(\rot_r(\positionerrpluseins)): u_{ij}\left(\sum_{k_1} u_{k_1j_1}\right)\ldots \left(\sum_{k_{r-1}} u_{k_{r-1}j_{r-1}}\right)=\left(\sum_{l_1} u_{i_1l_1}\right)\ldots \left(\sum_{l_{r+1}} u_{i_{r+1}l_{r+1}}\right)u_{ij}^*
\qquad\qquad\qquad\qquad\qquad\qquad\qquad\qquad\qquad\qquad\qquad\qquad\qquad\qquad\qquad
\]

\section{Definition of unitary easy quantum groups}\label{SectDefUnitary}

In this section we define unitary easy quantum groups. We first recall some notions and facts about compact matrix quantum groups.

\subsection{Compact matrix quantum group}\label{SectCMQG}

In \cite{woronowicz1987compact} (see also \cite{woronowicz1991remark}), Woronowicz gave a definition of a compact matrix quantum group within the $C^*$-algebraic framework.

\begin{defn}\label{DefCMQG}
A \emph{compact matrix quantum group}  for $n\geq 1$ is given by a unital $C^*$-algebra $A$ such that 
\begin{itemize}
\item the $C^*$-algebra $A$ is generated by $n^2$ elements $u_{ij}$, $1\leq i,j\leq n$,
\item the matrices $u=(u_{ij})$ and $\bar u=(u_{ij}^*)$ are invertible in $\mathcal{M}_{n}(A)$,
\item and the map $\Delta:A\to A\otimes_{\min}A$ given by $\Delta(u_{ij})=\sum_k u_{ik}\otimes u_{kj}$ is a $^*$-homomorphism.
\end{itemize}
\end{defn}

We sometimes also write $A=C(G)$ -- even if $A$ is noncommutative -- and speak of $G$ as the quantum group, sometimes specifying $(G,u)$ to indicate the matrix $u$ of the above definition. We then say that $(G, u^G)$ is a quantum subgroup of $(H,u^H)$, writing $G\subset H$, if there exists a $^*$-homomorphism $\phi:C(H)\to C(G)$ mapping $u_{ij}^H\mapsto u_{ij}^G$.  Note that there exist weaker definitions of quantum subgroups in the more general context of compact quantum groups. We should remark that the usual definition of compact matrix quantum groups requires $u$ and $u^t=(u_{ji})$ to be invertible, but since $\bar u^*=u^t$, our definition is an equivalent one.

In order to get a feeling for the above definition, observe that any compact group $G\subset M_n(\C)$ can be viewed as a compact matrix quantum group by considering $A=C(G)$, the algebra of continuous functions over $G$, and $u_{ij}:G\to\C$ the coordinate functions mapping a matrix $(g_{kl})\in G$ to the entry $g_{ij}$. Here, the map $\Delta$ is nothing but the dualized matrix multiplication. Thus, compact matrix quantum groups generalize compact matrix groups $G\subset M_n(\C)$.
Moreover, any compact matrix quantum group is a compact quantum group in the more general sense as defined by Woronowicz \cite{Wor98}.

The following definitions are due to Woronowicz \cite[Sect. 1]{woronowicz1987compact}.

\begin{defn}\label{DefW}
Let $(G,u)$ and $(H,v)$ be compact matrix quantum groups with $u$ and $v$ of size $n$. We say that $G$ and $H$ are
\begin{itemize}
\item[(a)] \emph{identical}, if there is a $*$-isomorphism from $C(G)$ to $C(H)$ mapping $u_{ij}$ to $v_{ij}$, hence mapping $u$ to $v$.
\item[(b)] \emph{similar}, if there is an invertible matrix $T\in GL_n(\C)$ and a $*$-isomorphism from $C(G)$ to $C(H)$ mapping $u$ to $TvT^{-1}$.
\end{itemize}
\end{defn}

We refer to \cite{woronowicz1987compact} and \cite{woronowicz1991remark} or to the surveys \cite{kustermans1999survey}, \cite{maes1998notes} and the books \cite{timmermann2008invitation}, \cite{neshveyev2013compact} for more information on the subject.

\subsection{An ad hoc definition of unitary easy quantum groups} \label{SectUnitaryEasy}

We first give an ad hoc definition of unitary easy quantum groups, which somehow seems to come out of the blue but will be put into a wider context in the next subsections. For the moment, recall that we associated certain relations $R(p)$ to a partition $p\in P^{\twocol}$, see Definition \ref{DefRel}.

\begin{defn}[easy QG -- ad hoc definition]\label{DefUnitaryEasy}
A compact matrix quantum group $G$ is called \emph{easy (in its full version)}, if there is a set of partitions $\CC_0\subset P^{\twocol}$ such that the $C^*$-algebra $C(G)$ associated to $G$ is the universal unital $C^*$-algebra generated by elements $u_{ij}$, $1\leq i,j\leq n$ such that $u$ and $\bar u$ are unitary and the $u_{ij}$ satisfy the relations $R(p)$ for $p\in \CC_0$.

The quantum group is called \emph{orthogonal easy (in its full version)} if in addition all $u_{ij}$ are selfadjoint.
\end{defn}

\begin{rem}\label{RemRelations}
One should check that any universal $C^*$-algebra of the form as in the above definition gives rise to a compact matrix quantum group in the sense of Definition \ref{DefCMQG}. To do so, it is straightforward to  verify that the relations $R(p)$ are satisfied by the elements $u_{ij}':=\sum_k u_{ik}\otimes u_{kj}$ whenever the $u_{ij}$ satisfy the relations $R(p)$. Note that as mentioned in Section \ref{SectCStar}, the relations making $u$ and $\bar u$ unitaries are also of the form $R(p)$ for some $p\in P^{\twocol}$.

Moreover, one can prove that if $A$ is the universal $C^*$-algebra generated by the $u_{ij}$ and the relations $R(p)$ for $p\in\CC_0$, together with those making $u$ and $\bar u$ unitaries, and if $B$ is the universal $C^*$-algebra  generated by elements $u_{ij}$ and the relations $R(p)$ for all partitions $p$ from the category $\CC=\langle\CC_0\rangle$ generated by the set $\CC_0$, then $A$ and $B$ coincide. Such considerations may also be found in an article in preparation by the second author.
\end{rem}

\begin{ex}\label{ExWang}
The following three quantum groups as defined by Wang \cite{wang1995free, wang1998quantum} are  easy  quantum groups:
\begin{itemize}
\item The \emph{free unitary quantum group} $U_n^+$ is given by the universal $C^*$-algebra generated by $u_{ij}$ and the relations making $u$ and $\bar u$ unitaries. It is easy with $\CC_0=\emptyset$.
\item The \emph{free orthogonal quantum group} $O_n^+$ is given by the quotient of $C(U_n^+)$ by the relations making the $u_{ij}$ selfadjoint. It is orthogonal easy with $\CC_0=\{\idpartwb\}$.
\item The \emph{free symmetric quantum group} $S_n^+$ is given by the quotient of $C(O_n^+)$ by the relations making the $u_{ij}$ projections with $\sum_k u_{ik}=\sum_k u_{kj}=1$. It is orthogonal easy with $\CC_0=\{\idpartwb,\singletonw,\downsingletonw,\vierpartrotwwww\}$.
\end{itemize}
\end{ex}

\begin{ex}
The symmetric group $S_n\subset M_n(\C)$ can be viewed as a compact matrix quantum group with algebra $C(S_n)$ and the coordinate functions $u_{ij}$ as generators (see Section \ref{SectCMQG}). It is easy with $\CC_0=\{\idpartwb,\singletonw,\downsingletonw,\vierpartrotwwww,\crosspartwwww\}$. We have $S_n\subset S_n^+\subset O_n^+\subset U_n^+$.
\end{ex}

\subsection{Tannaka-Krein for compact matrix quantum groups}

The above definition of easy quantum groups is motivated from the following result.
In \cite{woronowicz1988tannaka}, Woronowicz revealed an abstract way of obtaining compact matrix quantum groups from certain tensor categories. To be more precise: from concrete monoidal $W^*$-categories. Consider the following data:
\begin{itemize}
\item Let $R$ be a set endowed with a binary operation $\cdot:R\times R\to R, (r,s)\mapsto rs$.
\item For each $r\in R$, let $H_r$ be a finite-dimensional Hilbert space.
\item For each $r,s\in R$, let $\Mor(r,s)$ be a linear subspace of the space of all linear maps $T:H_r\to H_s$.
\end{itemize}

If now $\mathcal R:=(R,\cdot,(H_r)_{r\in R}, (\Mor(r,s))_{r,s\in R},)$ fulfills the following list of axioms, we say that $\mathcal R$ is a \emph{concrete monoidal $W^*$-category}:
\begin{itemize}
\item[(i)] $\mathcal R$ is closed under tensor products, i.e. if $T_i\in \Mor(r_i,s_i)$, for $i=1,2$, then also $T_1\otimes T_2\in\Mor(r_1r_2,s_1s_2)$ and $H_{rs}=H_r\otimes H_s$.
\item[(ii)] $\mathcal R$ is closed under composition, i.e. if  $T_i\in \Mor(r_i,s_i)$, for $i=1,2$, and $s_1=r_2$ then also $T_2T_1\in\Mor(r_1,s_2)$.
\item[(iii)] $\mathcal R$ is closed under involution, i.e. if $T\in \Mor(r,s)$, then $T^*\in\Mor(s,r)$.
\item[(iv)] The identity map $\id_r$ is in $\Mor(r,r)$, for all $r\in R$.
\item[(v)] If $H_r=H_s$ and $\id\in\Mor(r,s)$, then $r=s$.
\item[(vi)] We have $(rs)t=r(st)$ for all $r,s,t\in R$.
\item[(vii)] There is an element $1\in R$ such that $H_1=\C$ and $1r=r1=r$.
\end{itemize}

For $f\in R$ we denote by $\bar f\in R$ the element in  $R$ -- if it exists -- such that there is an invertible antilinear mapping $j: H_f\to H_{\bar f}$ with $t_j\in \Mor(1,f\bar f)$ and $\bar t_j\in \Mor(\bar ff,1)$, where for an orthonormal basis $(e_i)$ of $H_f$ and all $a\in H_{\bar f}, b\in H_f$, the maps $t_j$ and $\bar t_j$ are given by:
\[t_j(1)=\sum_i e_i\otimes j(e_i) \qquad\textnormal{and}\qquad \bar t_j(a\otimes b)=\langle j^{-1}(a),b\rangle_{H_f}\]
We say that $\{f,\bar f\}$ \emph{generates} $\mathcal R$ if for any $s\in R$ there exist $b_k\in \Mor(r_k,s)$, $k=1,\ldots,m$ such that $\sum b_kb_k^*=\id_s$ and every $r_k$ is a product of the form $r_k=f_1\ldots f_{n_k}$, where $f_i\in\{f,\bar f\}$.

Woronowicz also has a notion of completeness for a concrete monoidal $W^*$-category (using unitary equivalence, projections onto smaller objects, and direct sums), and he can show that any such category can be completed.

 His main theorem in \cite{woronowicz1988tannaka} is then that compact matrix quantum groups can be reconstructed from their intertwiner spaces. To be more precise, let $G$ be a compact matrix quantum group and let $R(G)$ be the class  of its finite-dimensional unitary representations. Every such representation $u^r\in R(G)$ is of the form $u^r=\sum_{i,j}e_{ij}\otimes u_{ij}^r\in B(H_r)\otimes C(G)$, where $e_{ij}$ are the matrix units in $B(H_r)$. We may equip $R(G)$ with the binary operation given by the tensor product of representations, $u^r\otimes u^s=\sum_{i,j,k,l}e_{ij}\otimes e_{kl}\otimes u_{ij}^ru_{kl}^s\in B(H_r\otimes H_s)\otimes C(G)$. Furthermore, we denote \emph{the space of intertwiners} by:
\[\Mor(u^r,u^s):=\{T:H_r\to H_s\;|\; Tu^r=u^sT\}\]
  We now formulate Woronowicz's theorem in the Kac case, i.e. when the conjugate of $u$ is given by $\bar u=(u_{ij}^*)$ (in general the form the conjugate of $u$ depends on the above map $j$).

\begin{thm}[{Tannaka-Krein for compact matrix quantum groups \cite{woronowicz1988tannaka}}]\label{ThmWorTK}\quad

\begin{itemize}
\item[(a)] Let $(G,u)$ be a compact matrix quantum group in the Kac case. The class $R(G)$ of all finite-dimensional unitary representations of $G$ is a complete concrete monoidal $W^*$-category generated by $\{u,\bar u\}$.
\item[(b)] Conversely, let $\mathcal R$ be a concrete monoidal $W^*$-category generated by $\{f,\bar f\}$. Then there is a compact matrix quantum group $(G,u)$ such that:
\begin{itemize}
\item[$\bullet$] $R(G)$ is a model of $\mathcal R$, i.e. for every $r\in R$ there is a representation $u^r\in R(G)$ with $u^r\in B(H_r)\otimes C(G)$, $u^{rs}=u^r\otimes u^s$, $Tu^r=u^sT$ for all $T\in\Mor(r,s)$, and $u^f=u$.
\item[$\bullet$] The completion of $\mathcal R$ coincides with $R(G)$.
\item[$\bullet$] The quantum group $G$ is universal in the sense, that any other model of $\mathcal R$ yields a quantum subgroup of $G$.
\end{itemize}
\end{itemize}
\end{thm}

Let us highlight two remarkable aspects of this theorem. Firstly, we obtain a compact matrix quantum group whenever we have a concrete monoidal $W^*$-category at hand. This is quite plausible in the case when our category is complete -- we define a quantum group simply by specifying all of its representations. But the second very nice thing is now, that we do not need to come up with \emph{all} representations but only with a rougher skeleton of it -- we only need a not necessarily complete category. This means, that the whole information about the quantum group is already given by a much smaller class of intertwiner spaces. This is a drastic reduction of complexity which will be needed in the sequel. In some sense, categories of partitions result from yet another reduction of complexity.

\subsection{A definition of unitary easy quantum groups using intertwiner spaces} \label{SectUnitaryEasy2}

Building on Woronowicz's Tannaka-Krein result, we may now give a more systematical definition of easy quantum groups using intertwiner spaces.  We very much follow the ideas of Banica and Speicher in the orthogonal case  \cite{banica2009liberation}. Let $p\in P^{\twocol}(k,l)$ be a two-colored partition. Using the numbers $\delta_p$ as defined in Definition \ref{DefDelta}, we define linear maps $T_p:(\C^n)^{\otimes k}\to(\C^n)^{\otimes l}$ by the following, where $(e_i)$ is the canonical orthonormal basis of $\C^n$.
\[T_p(e_{i_1}\otimes\ldots\otimes e_{i_k})
=\sum_{j_1,\ldots,j_l}\delta_p(i,j)e_{j_1}\otimes\ldots\otimes e_{j_l}\]
We use the convention that $(\C^n)^{\otimes 0}=\C$. Note that the coloring of the partition $p$ is not relevant for the definition of the map $T_p$. It only becomes relevant when $T_p$ is supposed to be an intertwiner. For this,  we need some more notation. For $r=(r_1,\ldots,r_k)\in\{\circ,\bullet\}^k$ we put $u^\circ:=u$, $u^\bullet:=\bar u$ and $u^r:=u^{r_1}\otimes\ldots\otimes u^{r_k}$.

\begin{defn}[easy QG -- intertwiner definition]\label{DefUnitaryEasy2}
A compact matrix quantum group $G$ with $S_n\subset G\subset U_n^+$ is called \emph{(unitary) easy}, if there is a category of partitions $\CC\subset P^{\twocol}$ such that for every $r\in\{\circ,\bullet\}^k$, every $s\in\{\circ,\bullet\}^l$, and all $k,l\in\N_0$, the space of intertwiners $\Mor(u^r,u^s)$ is spanned by all linear maps $T_p$ where $p$ is in $\CC(k,l)$ and has upper coloring according to $r$ and lower coloring according to  $s$. 

An easy quantum group $G$ is called an \emph{orthogonal easy quantum group}, if $G\subset O_n^+$.
\end{defn}

Note that if $r$ and $s$  consist only of white points, we have:
\[\Mor(u^r,u^s)=\Mor(u^{\otimes k},u^{\otimes l})=\spanlin\{T_p\;|\; p\in \CC(k,l), \textnormal{ all points of $p$ are white}\}\]

The above definition of an orthogonal easy quantum group is equivalent to the one given by Banica and Speicher \cite{banica2009liberation}. Indeed, let $G$ be an easy quantum group with $S_n\subset G\subset O_n^+$ as defined by Banica and Speicher. Hence, there is a category of non-colored partitions $\CC\subset P$ such that for all $k,l\in\N_0$:
\[\Mor(u^{\otimes k},u^{\otimes l})=\spanlin\{T_p\;|\; p\in \CC(k,l)\}\]
By Section \ref{SectOnecolored}, $\Psi^{-1}(\CC)\subset P^{\twocol}$ is a category of two-colored partitions containing $\paarpartww$. Now, $G$ is an easy quantum group in the sense of Definition \ref{DefUnitaryEasy2}, since the relations $R(\paarpartww)$ exactly give $u=\bar u$, see Section \ref{SectCStar}.
This shows that, our definition of unitary easy quantum groups is straightforward once the definition of orthogonal easy quantum groups was given by Banica and Speicher. Furthermore, the idea of unitary easy quantum groups was somehow around in the literature, see for instance \cite[Sect. 12]{banica2007hyperoctahedral}. Finally, we gratefully acknowledge the access to an unpublished draft by Banica, Curran and Speicher \cite{speicherunpublished} dating from 2012, in which the theory of unitary easy quantum groups was already developed. However, up to now, there was no honest definition of unitary easy quantum groups available in the literature.

The link between the $C^*$-algebraic relations and the intertwiner spaces is the following.

\begin{lem}\label{LemTpRp}
 Let $p\in P^{\twocol}(k,l)$ with upper color pattern $r\in\{\circ,\bullet\}^k$ and lower color pattern $s\in\{\circ,\bullet\}^l$, and let $n\in\N$. Let $A$ be a $C^*$-algebra generated by elements $u_{ij}$, $1\leq i,j\leq n$. Then, the generators $u_{ij}$ fulfill the relations $R(p)$ if and only if $T_pu^r=u^s T_p$ for $u^r$ and $u^s$ as defined above.
\end{lem}
\begin{proof}
Applying $u^r=\sum_{i,j} e_{j_1i_1}\otimes \ldots\otimes e_{j_ki_k}\otimes u_{j_1i_1}^{r_1}\ldots u_{j_ki_k}^{r_k}\in M_n(\C)^{\otimes k}\otimes A$ to a vector $e_{i_1}\otimes\ldots\otimes e_{i_k}\otimes 1$ amounts to:
\[u^r \left(e_{i_1}\otimes\ldots\otimes e_{i_k}\otimes 1\right)=\sum_{\alpha_1,\ldots,\alpha_k} e_{\alpha_1}\otimes\ldots\otimes e_{\alpha_k}\otimes u_{\alpha_1 i_1}^{r_1}\ldots u_{\alpha_ki_k}^{r_k}\]
Thus:
\[T_pu^r\left(e_{i_1}\otimes\ldots\otimes e_{i_k}\otimes 1\right)=\sum_{\alpha,\beta}\delta_p(\alpha,\beta)e_{\beta_1}\otimes\ldots\otimes e_{\beta_l}\otimes u_{\alpha_1 i_1}^{r_1}\ldots u_{\alpha_ki_k}^{r_k}\]
Whereas:
\[u^sT_p\left(e_{i_1}\otimes\ldots\otimes e_{i_k}\otimes 1\right)=\sum_{\gamma,\beta}\delta_p(i,\gamma)e_{\beta_1}\otimes\ldots\otimes e_{\beta_l}\otimes u_{\beta_1 \gamma_1}^{s_1}\ldots u_{\beta_l\gamma_l}^{s_l}\]
Comparison of the coefficients yields the claim.
\end{proof}

\subsection{From categories of partitions to easy quantum groups via Tannaka-Krein}

By Definition \ref{DefUnitaryEasy2}, an  easy quantum group is a quantum group, whose intertwiner spaces are indexed by elements from a category of partitions. But can we always assign a quantum group to a category of partitions? We will now show how to construct a concrete monoidal $W^*$-category given a category of partitions. We then  apply Woronowicz's Tannaka-Krein result and we are done. This yields a third definition of easy quantum groups.

The first step is to show -- like in \cite{banica2009liberation} -- that operations on the maps $T_p$ match nicely with operations on partitions. 
To do this for the composition of two partitions $q\in P^{\twocol}(k,l)$ and $p\in P^{\twocol}(l,m)$, we need some more notation. Recall from Section \ref{SectOperations} that the $l$ middle points are removed when forming the composition $pq\in P^{\twocol}(k,m)$. But before doing so, we restrict the partition $p$ to its $l$ upper, and $q$ to its $l$ lower points. We then form the maximum of these two restricted partitions $p_0$ and $q_0$, i.e. we consider the minimal partition $r$ such that every block of $p_0$ is contained in a block of $r$, and likewise for $q_0$. Now let $\textnormal{rl}(q,p)$ be the number of those blocks of $r$ containing only points which are neither connected to one of the $k$ upper points of $q$ nor to one of the $m$ lower points of $p$.

\begin{lem}\label{LemPropTp}
The assignment $p\mapsto T_p$ respects the category operations for partitions:
\begin{itemize}
\item[(a)] $T_p\otimes T_q=T_{p\otimes q}$
\item[(b)] $T_qT_p=n^{\textnormal{rl}(q,p)}T_{qp}$ 
\item[(c)] $(T_p)^*=T_{p^*}$
\end{itemize}
\end{lem}
\begin{proof}
The proof is word by word the same  as in \cite[Prop 1.9]{banica2009liberation} (since $T_p$ does not see the coloring of $p$).
\end{proof}
In fact, it is not important to understand the precise number $n^{\textnormal{rl}(q,p)}$ -- the only thing that matters is that the span of all $T_p$ is closed under composition of maps.

We are now ready to apply Woronowicz's theorem.

\begin{prop}\label{PropTK}
Let $\CC\subset P^{\twocol}$ be a category of partitions. We may associate a concrete monoidal $W^*$-category with generating set $\{f,\bar f\}$ to it, and hence a compact matrix quantum group $G$ of Kac type (in the sense that $u^f=u=(u_{ij})$ and $u^{\bar f}=\bar u=(u_{ij}^*)$). The $C^*$-algebra $C(G)$ is the universal $C^*$-algebra generated by the relations $R(p)$ with $p\in\CC$, thus it is  easy in the sense of Definition \ref{DefUnitaryEasy}. 
\end{prop}
\begin{proof}
This is how we associate a concrete monoidal $W^*$-category $\mathcal R_{\CC}$ to a category $\CC$ of partitions:
\begin{itemize}
\item We let $R_0$ be the set of all words $r=r_1\ldots r_k\in\{\circ,\bullet\}^k$, for all $k\in\N_0$. Denote by $|r|:=k$ the length of the word. We also consider the empty word $e$ with $|e|=0$. We let $R$ be the set of words in $R_0$ with the identification of  two words $r$ and $s$ whenever $|r|=|s|$ and the tensor product of $\idpartww$, $\idpartwb$, $\idpartbw$, and $\idpartbb$ with upper color pattern $r$ and lower color pattern $s$ is in $\CC$.
We define a binary operation $\cdot:R\times R\to R$ by concatenation of words.
\item For $r\in R$, put $H_r:= (\C^n)^{\otimes|r|}$. Moreover $H_e:=\C$.
\item Let $\Mor(r,s)$ be the linear span of all maps $T_p$ where $p\in \CC(|r|,|s|)$ such that the upper points of $p$ are colored according to $r$ and the lower points according to $s$.
\item Put $f:=\circ$ and $\bar f:=\bullet$.
\end{itemize} 

Using Lemma \ref{LemPropTp}, conditions (i), (ii) and (iii) of a concrete monoidal $W^*$-category are readily verified. The map for (iv) is given by a suitable tensor product of the identity partitions $\idpartww$ and $\idpartbb$, and (v) is fulfilled since we passed from $R_0$ to $R$. Finally, the concatenation of words is associative, thus (vi) is satisfied und as for (vii), we use the empty word. 
Moreover, $f=\circ$ and $\bar f=\bullet$ are conjugate with the invertible antilinear mapping $j:\C^n\to\C^n$ given by $\sum \alpha_i e_i\mapsto \sum\bar\alpha_i e_i$. Then $t_j=T_{\paarpartwb}\in \Mor(1,f\bar f)$ and $\bar t_j=T_{\baarpartbw}\in \Mor(\bar ff,1)$. Furthermore, for any $s\in R$, we let $b_1\in \Mor(s,s)$ be the identity map, proving that $\{f,\bar f\}$ generates the $W^*$-category. Hence, $\mathcal R_{\CC}$ is a concrete monoidal $W^*$-category and we obtain a quantum group $G$ due to Theorem \ref{ThmWorTK}.

The $C^*$-algebra $C(G)$  fulfills the relations $R(p)$ for all $p\in\CC$ by Lemma \ref{LemTpRp}. On the other hand, the universal $C^*$-algebra generated by all relations $R(p)$, $p\in\CC$ gives rise to a model for $\mathcal R_{\CC}$. Hence, it coincides with $C(G)$. 
\end{proof}

In this sense, the theory of easy quantum groups arises as a very direct application of Woronowicz's Tannaka-Krein result and categories of partitions are yet another reduction of complexity compared to concrete monoidal $W^*$-categories (see the discussion after Theorem \ref{ThmWorTK}). Implicitely first roots of such a diagrammatic approach to these $W^*$-categories may be found in Woronowicz's construction of the twisted $SU(N)$ groups \cite[Proof of Thm. 1.4]{woronowicz1988tannaka}.

\begin{cor}
If $G$ is easy in the sense of Definition \ref{DefUnitaryEasy}, it is easy in the sense of Definition \ref{DefUnitaryEasy2}.
\end{cor}
\begin{proof}
Let $\CC_0\subset P^{\twocol}$ be a set and let $C(G)$ be the universal $C^*$-algebra generated by elements $u_{ij}$ with $u$ and $\bar u$ unitary and the relations $R(p)$, $p\in\CC_0$. By Remark  \ref{RemRelations}, we may assume that $\CC_0$ is a category of partitions and by Proposition \ref{PropTK} we obtain a compact matrix quantum group $H$ such that $C(H)=C(G)$. Hence $G=H$ and its intertwiners are given by the span of all $T_p$, $p\in\CC_0$. Moreover, as $u$ and $\bar u$ are unitaries, we have $G\subset U_n^+$. Finally, $P^{\twocol}$ is the category generated by $\{\idpartwb,\singletonw,\downsingletonw,\vierpartrotwwww,\crosspartwwww\}$ (see \cite{weber2013classification}) and thus we have a surjection of $C(G)$ to $C(S_n)$ mapping generators to generators. This proves $S_n\subset G$ and we are done.
\end{proof}

Asking for the converse, we have to mention one subtlety. Given a $*$-algebra $A$, there may exist several $C^*$-completions of it. A compact matrix quantum group $G$ is in its \emph{full version}, if $C(G)$ is the universal enveloping $C^*$-algebra $C_f(G)$ of the $*$-algebra $\textnormal{Pol}(G)$ generated by polynomials in the elements $u_{ij}$ for $i,j\in\{1,\ldots,n\}$. We also have a \emph{reduced version} of $G$ given by the $C^*$-algebra $C_r(G)$ obtained from the GNS construction with the Haar state on $C_f(G)$ (see \cite{neshveyev2013compact} for details). This is yet another completion of $\textnormal{Pol}(G)$. A quantum group in its full version, in its reduced version or in its algebraic version (taking $\textnormal{Pol}(G)$ as the algebra) -- from an algebraic point of view, these are all \emph{different} objects since the underlying algebras are different. But on a more abstract level, it is always the \emph{same} quantum group -- in different disguises, but with the same ``quantum group features'' like the representation theory or the intertwiner spaces.

Now, Definition \ref{DefUnitaryEasy} and Woronowicz's Tannaka-Krein approach only cover quantum groups in their full versions while Definition \ref{DefUnitaryEasy2} appears to be more general. However, any easy quantum group in the sense of Definition \ref{DefUnitaryEasy2} will be easy in the sense of Definition \ref{DefUnitaryEasy} once put into its full version, i.e. once its algebra $\textnormal{Pol}(G)$ is completed in a maximal way. We conclude that all three approaches to a definition of easy quantum groups yield the same objects linking combinatorics with quantum algebra in the following nice way.

\begin{cor}
Categories of partitions are in one-to-one correspondence with easy quantum groups. More precisely, to a category $\CC$, we may assign an easy quantum group $G_n$ for every $n\in\N$, as seen in  Proposition \ref{PropTK}. Now, the correspondence 
\[\{\CC\subset P^{\twocol} \textnormal{ categ. of part.}\} \quad\longleftrightarrow\quad \{(G_n)_{n\in\N} \textnormal{ easy q.g. with same categ. of part.}\}\]
 is one-to-one.
\end{cor}
\begin{proof}
Surjectivity is clear by the definition of easy quantum groups. As for injectivity, assume that two categories $\CC$ and $\DD$ yield the same sequence $(G_n)$ of easy quantum groups. Let $p\in \CC(k,l)$ and let $n\geq k+l$. Then the associated linear map $T_p$ is an intertwiner for $G_n$, since $G_n$ is easy with category $\CC$. Since it is also easy with $\DD$, the map $T_p$ is in the span of all maps $T_q$ for $q\in\DD(k,l)$. Now, the maps $T_r$ for $r\in P^{\twocol}(k,l)$ are linearly independent since $n\geq k+l$ (see \cite[Th. 1.10]{banica2009liberation} or  \cite[Sect. 3]{bergeron2008invariants}). Thus, $T_p$ must coincide with one of the $T_q$, $q\in \DD(k,l)$ and hence $p=q\in\DD$.
\end{proof}

\begin{rem}
As discussed in Section \ref{SectCMQG}, the  $W^*$-category appearing in Proposition \ref{PropTK} does not necessarily give all representations of the easy  quantum group, since it might not be complete. A study of the completion of this $W^*$-category is implicitely given in the article \cite{freslonweber}.
\end{rem}

\section{Free unitary easy quantum groups}\label{SectFreeUnitary}

In this section we describe the easy quantum groups corresponding to categories of noncrossing partitions. Our main result may be found in Section \ref{SectMain}.

\subsection{Definition of the free case}

\begin{defn}
A \emph{free easy quantum group} is an easy quantum group $G$ such that the corresponding category of partitions $\mathcal C\subset P^{\twocol}$ contains only noncrossing partitions, thus $\CC\subset NC^{\twocol}$.
\end{defn}

The terminology goes back to Wang's papers \cite{wang1995free, wang1998quantum} and more explicitely to  \cite{banica2007hyperoctahedral}. See also \cite{freslon2014partition} for a general definition of a free compact matrix quantum group, tracing back to \cite{banica2009fusion}. In principle,
Theorems \ref{ThmClassiGlob} and \ref{ThmClassiLoc} yield the complete list of free easy quantum groups. However, their interplay with other quantum groups might not be so clear right away. In particular, it is interesting to know how ``new'' or exotic they are. We will now reveal how these free easy quantum groups may be constructed from some basic quantum groups using products with reflection groups as explained below.

\subsection{Free and tensor complexifications with $\Z_d$}\label{SectProd1}

In \cite{wang1995tensor} and \cite{wang1995free}, Wang proved the existence of a comultiplication on the free product as well as on the tensor  product of the $C^*$-algebras associated to quantum groups (see also \cite{timmermann2008invitation}). More precisely, let $G$ and $H$ be two compact (matrix) quantum groups with comultiplications $\Delta_G$ resp. $\Delta_H$. Let $C(G)\square C(H)$ either be the unital free product $C(G)*C(H)$ of the two $C^*$-algebras or the maximal tensor product $C(G)\otimes_{\max}C(H)$. Denote by $\iota_{C(G)}$ the embedding of $C(G)$ into $C(G)\square C(H)$ and likewise by $\iota_{C(G)\square C(G)}$ the embedding of $C(G)\otimes_{\min} C(G)$ into $(C(G)\square C(H))\otimes_{\min}(C(G)\square C(H))$.

\begin{prop}[\cite{wang1995tensor,wang1995free}]\label{PropWangFreeTens}
Given two compact (matrix) quantum groups $G$ and $H$, there is always a comultiplication $\Delta$ on $C(G)\square C(H)$ for $\square\in\{*,\otimes_{\max}\}$ such that:
\[\Delta\circ\iota_{C(G)}=\iota_{C(G)\square C(G)}\circ \Delta_{G} \qquad \textnormal{and}\qquad \Delta\circ\iota_{C(H)}=\iota_{C(H)\square C(H)}\circ \Delta_{H}\]
\end{prop}

As a consequence, one can define the free product and the direct product of compact matrix quantum groups. The fundamental corepresentation is then given by the direct sum of these representations, thus by $\begin{pmatrix} u&0\\0&v\end{pmatrix}$, where $u$ and $v$ are the matrices of generators for $G$ resp. $H$. We now define another kind of free resp. tensor product of two compact matrix quantum groups. 

\begin{rem}\label{RemUniv}
Recall that for unital $C^*$-algebras $A$ and $B$, the maximal tensor product $A\otimes_{\max} B$ can be seen as the universal $C^*$-algebra generated by elements $a\in A$ (with the relations of $A$) and $b\in B$ (with the relations of $B$) such that all such $a$ and $b$ commute; the units of $A$ and $B$ are identified. We thus simply write $ab$ for elements $a\otimes b$. The free product $A*B$ in turn is the universal $C^*$-algebra generated by elements $a\in A$ (with the relations of $A$) and $b\in B$ (with the relations of $B$) without any further relations; again the units of $A$ and $B$ are identified.
\end{rem}

\begin{defn}
Let $(G,u)$ and $(H,v)$ be two compact matrix quantum groups with $u$ of size $n$ and $v$ of size $m$.
\begin{itemize}
\item[(a)] The \emph{glued free product $G\freeglued H$ of $G$ and $H$} is given by the $C^*$-subalgebra $C^*(u_{ij}v_{kl}\;|\; 1\leq i,j\leq n \textnormal{ and } 1\leq k,l\leq m)\subset C(G)*C(H)$.
\item[(b)] The \emph{glued direct product $G\tensorglued H$ of $G$ and $H$} is given by the $C^*$-subalgebra $C^*(u_{ij}v_{kl}\;|\; 1\leq i,j\leq n \textnormal{ and } 1\leq k,l\leq m)\subset C(G)\otimes_{\max}C(H)$.
\end{itemize}
\end{defn}

As a simple consequence of Wang's result, the glued free product and the glued direct product are again compact matrix quantum groups.

\begin{cor}\label{CorGlued}
The $C^*$-subalgebra $C^*(u_{ij}v_{kl}\;|\; 1\leq i,j\leq n \textnormal{ and } 1\leq k,l\leq m)$ of $C(G)\square C(H)$,  $\square\in\{*,\otimes_{\max}\}$  admits a comultiplication $\Delta(u_{ij}v_{kl})=\Delta_G(u_{ij})\Delta_H(v_{kl})$.
\end{cor}
\begin{proof}
Restriction of the comultiplication $\Delta$ of Proposition \ref{PropWangFreeTens} yields the result.
\end{proof}

Note that in general, the glued products differ from Wang's products, i.e. we have $G*H\neq G\freeglued H$ and $G\times H\neq G\tensorglued H$. The crucial difference lies in the fundamental matrix $u\oplus v$ (Wang) versus $u\otimes v$ (glued).

To a discrete group $\Gamma$, one can associate the universal $C^*$-algebra $C^*(\Gamma)$ generated by unitaries $u_g$, $g\in \Gamma$ with $u_gu_h=u_{gh}$, $u_g^*=u_{g^{-1}}$. It is well known that the comultiplication $\Delta(u_g)=u_g\otimes u_g$ turns it into a compact quantum group denoted by $\widehat\Gamma$.

\begin{cor}
Let $\Gamma$ be a discrete group generated by a single element $g_0$, and denote by $z$ the generator $u_{g_0}$ of $C^*(\Gamma)$. Let $(G,u)$ be a compact matrix quantum group. Then $G\freeglued \widehat\Gamma$ and $G\tensorglued\widehat\Gamma$ are compact matrix quantum groups given by $C^*(u_{ij}z)$ in $C(G)* C^*(\Gamma)$ resp. in $C(G)\otimes_{\max} C^*(\Gamma)$ and $\Delta(u_{ij}z)=\sum_k u_{ik}z\otimes u_{kj}z$.
\end{cor}
\begin{proof}
Since $\Gamma$ is generated by a single element, $(C^*(\Gamma),z)$ is a compact matrix quantum group with matrix of size 1. Using Corollary \ref{CorGlued} we obtain the result.
\end{proof}

Denote by $\Z_d$ the cyclic group $\Z_d:=\Z / d\Z$. Note that $C^*(\Z_d)$ is the universal $C^*$-algebra generated by a unitary $z$ such that $z^d=1$. In order to simplify the notation, we will also write $\Z_0$ for $\Z$.

\begin{defn}
Let $G$ be a compact matrix quantum group.
\begin{itemize}
\item[(a)] The quantum group $G\freeglued\widehat{\Z_d}$ is called the \emph{free $d$-complexification of $G$} and $G\freeglued\widehat{\Z}$ is called the \emph{free complexification}.
\item[(b)] The quantum group $G\tensorglued\widehat{\Z_d}$ is called the \emph{tensor $d$-complexification of $G$} and $G\tensorglued\widehat{\Z}$ is called the \emph{tensor complexification}.
\end{itemize}
\end{defn}

The above definition is a generalization of Banica's free complexification \cite{banica2008note}, \cite{banica1999representations}.

\begin{lem}\label{LemNonEasyMaybe}
Let $G\subset U_n^+$ be a compact matrix quantum group. Then $G\subset G\freeglued\widehat{\Z_d}\subset U_n^+$ and $G\subset G\tensorglued\widehat{\Z_d}\subset U_n^+$ for $d\in\N_0$. Note that if $G\subset O_n^+$, we do not have $G\freeglued\widehat{\Z_d}\subset O_n^+$ or $G\tensorglued\widehat{\Z_d}\subset O_n^+$ in general.
\end{lem}
\begin{proof}
As before, let  $\square=*$ or $\square=\otimes_{\max}$. We have $C(G)\square C^*(\Z_d)\to C(G)$ mapping $u_{ij}\mapsto u_{ij}$ and $z\mapsto 1$. This yields $G\subset G\square\widehat{\Z_d}$. As for $G\square\widehat{\Z_d}\subset U_n^+$, it is straightforward to check that the matrices $u':=(u_{ij}z)$ and $\bar u'=(z^*u_{ij}^*)$ are unitaries. As for the remark on the orthogonal case, note that $u_{ij}z$ does not need to be selfadjoint, even if $u_{ij}$ is.
\end{proof}

\subsection{The $r$-selfadjoint free $d$-complexification}\label{SectProd2}

We need to introduce yet another product of a quantum group with $\widehat{\Z_d}$.

\begin{defn}
Let $(G,u)$ be a compact matrix quantum group with $u=(u_{ij})$ of size $n$, and let $r,d\in\N$. Denote the generator of $C^*(\Z_d)$ by $z$. Let $A$ be the quotient of $C(G)*C^*(\Z_d)$ by the ideal generated by $u_{ij}z^r=(z^*)^ru_{ij}^*$ for all $i,j$.
The \emph{$r$-selfadjoint free $d$-complexification} $G\freeglued_r \widehat{\Z_d}$ is given by the $C^*$-subalgebra $C^*(u_{ij}z, 1\leq i,j\leq n)$ of $A$.
\end{defn}

In a certain sense, the above product is a free $d$-complexification with glimpses of selfadjointness -- the elements $u_{ij}z$ are not self-adjoint in general but $u_{ij}z^r$ is. Note that from the definition it is not clear whether any of the elements $u_{ij}z^r$ lies in $C^*(u_{ij}z, 1\leq i,j\leq n)$.

\begin{prop}
The $r$-selfadjoint free $d$-complexification gives rise to a compact matrix quantum group: There is a $*$-homomorphism $\Delta:A\to A\otimes_{\min} A$ on the $C^*$-algebra $A$ of the above definition sending $u_{ij}$ to $\sum_k u_{ik}\otimes u_{kj}$ and $z$ to $z\otimes z$. It restricts to $C^*(u_{ij}z, 1\leq i,j\leq n)\to C^*(u_{ij}z, 1\leq i,j\leq n) \otimes_{\min} C^*(u_{ij}z, 1\leq i,j\leq n)$ turning $C^*(u_{ij}z, 1\leq i,j\leq n)$ into a compact matrix quantum group.
\end{prop}
\begin{proof}
By Wang's result, Proposition \ref{PropWangFreeTens}, there is a comultiplication $\Delta_0$ on $B:=C(G)*C^*(\Z_d)$. Composing it with the quotient map $\pi\otimes\pi:B\otimes B\to A\otimes A$, we obtain a $*$-homomorphism from $B$ to $A\otimes A$ mapping $u_{ij}$ to $u_{ij}':=\sum_k u_{ik}\otimes u_{kj}$ and $z$ to $z':=z\otimes z$. Since $u_{ij}'$ and $z'$ satisfy $u_{ij}'z'^r=(z'^*)^ru_{ij}'^*$, the map $(\pi\otimes\pi)\circ\Delta_0$ factors through $\pi:B\to A$ and we obtain a map $\Delta:A\to A\otimes A$ as desired.
\end{proof}

In principle one could also define an $r$-selfadjoint tensor $d$-complexification, but we do not need it for our purpose.

\subsection{Free wreath products}\label{SectProd3}

Finally, we also need a product construction provided by Bichon \cite{bichon2004free}. 
Denote by $z\mapsto z^{(i)}$ the embedding of $C^*(\Z_d)$ into the $i$-th copy of $C^*(\Z_d)$ in the $n$-fold free product $C^*(\Z_d)^{*n}$.

\begin{defn}[{\cite{bichon2004free}}]
Let $(G,u)$ be a compact matrix quantum group. The \emph{free wreath product}  $\widehat{\Z_d}\wr_*S_n^+$ is given by  the quotient of $C^*(\Z_d)^{*n}*C(S_n^+)$ by the ideal generated by $u_{ij}z^{(i)}=z^{(i)}u_{ij}$. The fundamental unitary is given by $(u_{ij}z^{(i)})_{i,j}$.
\end{defn}

Strictly following Bichon's article, the free wreath product comes without a specification of the fundamental unitary in the case of compact matrix quantum groups. The one that is commonly used nowadays is the one above  (see for instance \cite{wahl}, \cite{lemeuxfusion} or \cite{banica2009fusion}). In this sense, the free wreath product is already in a ``glued version''. Note that due to the relations of $C(S_n^+)$, the $C^*$-subalgebra generated by the elements $u_{ij}z^{(i)}$ coincides with the whole $C^*$-algebra, so the question of how to endow the $C^*$-algebra in the above definition with a compact matrix quantum group structure lies only in the definition of the generating unitary. See \cite{bichon2004free} for a proof that the free wreath product is a compact matrix quantum group.

\subsection{Construction of free easy quantum groups using $d$-complexifications}\label{SectMain}

Using the products  introduced in the previous sections, we may now construct the free easy quantum groups out of a few basic ones. This constitutes the main theorem of this article. Before doing so, we recall the seven free orthogonal easy quantum groups as classified in \cite{weber2013classification}.

\begin{rem}\label{Reminder}
Recall that a category $\CC\subset NC$ of non-colored partitions can be viewed as a category of colored partitions containing $\idpartwb$ (or equivalently $\paarpartww$), see Section \ref{SectOnecolored}. We now list the seven free orthogonal easy quantum groups and their corresponding categories as subsets of $NC^{\twocol}$. They appear in the list of Theorem \ref{ThmClassiGlob}.
For the convenience of the reader, we also list the $C^*$-algebraic relations, besides those making the $u_{ij}$ selfadjoint and $u$ orthogonal (i.e. $\sum_k u_{ik}u_{jk}=\sum_k u_{ki}u_{kj}=\delta_{ij}$).
\begin{itemize}
\item $O_n^+$: $\categ{\OOO}{\glob}{2}=\langle\idpartwb\rangle$, no further relations. 
\item $H_n^+$: $\categ{\HHH}{\glob}{2}=\langle\idpartwb,\vierpartwbwb\rangle$,  $u_{ik}u_{jk}=u_{ki}u_{kj}=0$ if $i\neq j$.
\item $S_n^+$: $\categ{\SSS}{\glob}{1}=\langle\idpartwb,\vierpartwbwb,\singletonw\rangle$,  $u_{ij}=u_{ij}^*=u_{ij}^2$,  $\sum_k u_{ik}=\sum_k u_{kj}=1$.
\item ${S'_n}^+$: $\categ{\SSS}{\glob}{2}=\langle\idpartwb,\vierpartwbwb,\singletonw\otimes\singletonb\rangle$,  $u_{ik}u_{jk}=u_{ki}u_{kj}=0$ if $i\neq j$, $\sum_k u_{ik}=\sum_k u_{kj}$ independent of $i,j$.
\item $B_n^+$: $\categ{\BBB'}{\glob}{1}=\langle\idpartwb,\singletonw\rangle$,   $\sum_k u_{ik}=\sum_k u_{kj}=1$.
\item ${B'_n}^+$: $\categ{\BBB'}{\glob}{2}=\langle\idpartwb,\positionerwwbb\rangle$,  $s:=\sum_k u_{ik}=\sum_k u_{kj}$ independent of $i,j$ and $u_{ij}s=su_{ij}$.
\item $B_n^{\#+}$: $\categ{\BBB}{\glob}{2}=\langle\idpartwb,\singletonw\otimes\singletonb\rangle$, $\sum_k u_{ik}=\sum_k u_{kj}$ independent of $i,j$.
\end{itemize}
\end{rem}

Using the list of Theorem \ref{ThmClassiGlob}, we obtain the following theorem which will be proven in the next subsection. We use the identification of compact matrix quantum groups as described in Definition \ref{DefW}(a).

\begin{thm}[globally colorized case]\label{ThmEasyGlob}
The globally colorized categories give rise to the following quantum groups.

\begin{tabular}{lll}
$\categ{\OOO}{\glob}{k}$: & $O_n^+\tensorglued \widehat{\Z_k}$, &$k\in 2\N_0$\\
$\categ{\HHH}{\glob}{k}$: & $H_n^+\tensorglued \widehat{\Z_k}$, &$k\in 2\N_0$\\
$\categ{\SSS}{\glob}{k}$: & $S_n^+\tensorglued \widehat{\Z_k}$, &$k\in \N_0$\\
$\categ{\BBB}{\glob}{k}$: & $B_n^{\#+}\tensorglued \widehat{\Z_k}$, &$k\in 2\N_0$\\
$\categ{\BBB'}{\glob}{k}$: & $B_n^+\tensorglued \widehat{\Z_k}$, & $k\in \N_0$
\end{tabular}

In the case $k=0$, we replace $\Z_k$ by $\Z$. For $k=1$, we have $\Z_1=\{e\}$, hence we then omit the product with $\widehat{\Z_1}$.
\end{thm}

\begin{rem}\label{RemConfusionWithProducts}
Comparing Remark \ref{Reminder} with Theorem \ref{ThmEasyGlob}, we observe that the relations of $\categ{\OOO}{\glob}{2}$ give rise to $O_n^+$, whereas Theorem \ref{ThmEasyGlob} yields  $O_n^+\tensorglued \widehat{\Z_2}$ in this case. The solution is simply that $O_n^+$ is isomorphic to $O_n^+\tensorglued \widehat{\Z_2}$, see Proposition \ref{OnIsom} below. The same happens for $\categ{\HHH}{\glob}{2}$ and $\categ{\BBB}{\glob}{2}$. In the cases of $\categ{\SSS}{\glob}{2}$ and $\categ{\BBB'}{\glob}{2}$, we recover the well-known facts ${S'_n}^+\cong S_n^+\tensorglued\widehat{\Z_2}$ and ${B'_n}^+\cong B_n^+\tensorglued\widehat{\Z_2}$ from \cite{raum} and \cite{weber2013classification}.
\end{rem}

In order to state the theorem in the locally colorized case, we have to define a new quantum group. It is a kind of a unitary version of $B_n^+$.
The credit for discovering it goes to the unpublished draft \cite{speicherunpublished} by Banica, Curran and Speicher.

\begin{defn}
Let $C_n^+$ be the quantum group given by the universal $C^*$-algebra generated by $u_{ij}$, $1\leq i,j\leq n$, such that $u$ and $\bar u$ are unitaries and $\sum_ku_{ik}=\sum_ku_{kj}=1$ for all $i,j$.
\end{defn}

It can be read directly from the relations in Section \ref{SectCStar} that $C_n^+$ is free easy with category $\categ{\BBB}{\loc}{1,0}$. The next theorem is based on the classification in  Theorem \ref{ThmClassiLoc}. Recall that $d\vert k$ stands for the fact that $k$ is a multiple of $d$. Thus, if $d=0$, then also $k=0$.

\begin{thm}[locally colorized case]\label{ThmEasyLoca}
The locally colorized categories give rise to the following quantum groups. the case $d=k$ includes $d=0$.

\begin{tabular}{lll}
$\categg{\OOO}{\loc}$: &$O_n^+\freeglued \widehat{\Z}=U_n^+$\\
 $\categg{\HHH'}{\loc}$: &$H_n^+\freeglued \widehat{\Z}$\\
 $\categ{\HHH}{\loc}{k,d}$: & $(\widehat{\Z_d}\wr_*S_n^+)\tensorglued \widehat{\Z_k}$, &$k,d\in\N_0\backslash\{1,2\}, d|k, d\neq0$\\
 & $\widehat{\Z_k}\wr_*S_n^+$ & ($d=k$)\\
$\categ{\SSS}{\loc}{k,d}$: &$(S_n^+\freeglued \widehat{\Z_d}) \tensorglued \widehat{\Z_k}$, & $k,d\in\N_0\backslash\{1\}, d|k, d\neq 0$\\
&$S_n^+\freeglued\widehat{\Z_k}$ & ($d=k$)\\
$\categ{\BBB}{\loc}{k,d}$: &$(C_n^+\freeglued \widehat{\Z_d}) \tensorglued \widehat{\Z_k}$, &
$k,d\in\N_0, d|k, d\neq 0$\\
& $C_n^+\freeglued \widehat{\Z_k}$ &  ($d=k$)\\
$\categ{\BBB'}{\loc}{k,d,0}$: &$(B_n^+\freeglued \widehat{\Z_d}) \tensorglued \widehat{\Z_k}$, &
$k,d\in\N_0\backslash\{1\}, d|k, d\neq 0$\\
& $B_n^+\freeglued \widehat{\Z_k}$ & ($d=k$)\\
 $\categ{\BBB'}{\loc}{k,d,r}$: &$(C_n^+\freeglued_r \widehat{\Z_d}) \tensorglued \widehat{\Z_k}$, &
$k\in\N_0\backslash\{1\}, d\in2\N_0\backslash\{0,2\}, d|k, r=\frac{d}{2}$\\
& $C_n^+\freeglued_r \widehat{\Z_k} $ & ($d=k$)\\
\end{tabular}

In the case $k=0$, we replace $\Z_k$ by $\Z$.
\end{thm}

Before passing to the proofs, let us remark a few things.

\begin{rem}\label{RemXYZ}
The cases $d\neq 0$ and $d=k$ do not exclude each other -- we might have $d=k\neq 0$. We thus proved incidentally $(G\freeglued \widehat{\Z_k}) \tensorglued \widehat{\Z_k}\cong G\freeglued \widehat{\Z_k}$ etc. See also Proposition \ref{OnIsom} for a direct proof.
\end{rem}

\begin{rem}
In \cite{banicaUn}, Banica observed that $U_n^+$ can be written as a free complexification of $O_n^+$. In \cite{banica2008note}, he studied the free complexifications of  $H_n^+$ and $S_n^+$, denoting them by $K_n^+$ and $P_n^+$. The former one also appeared in \cite{banica2011free}.
All these quantum groups are free easy.
\end{rem}

\begin{rem}
In \cite{raum} and \cite{weber2013classification}, it is proven that the $C^*$-algebra $C(B_n^{\#+})$ is isomorphic to $C(B_n^+)*C^*(\Z_2)$. But this isomorphism does \emph{not} map $u_{ij}\mapsto u_{ij}z$, so it does \emph{not} provide an isomorphism of $B_n^{\#+}$ and $B_n^+\freeglued\widehat{\Z_2}$. In fact, the elements $u_{ij}z$ in $C(B_n^+\freeglued\widehat{\Z_2})$ are not selfadjoint, so $B_n^+\freeglued\widehat{\Z_2}$ happens to be no quantum subgroup of $O_n^+$ while $B_n^{\#+}$ is.
\end{rem}

\subsection{The proofs of Theorem \ref{ThmEasyGlob} and Theorem \ref{ThmEasyLoca}}

The proofs of Theorem \ref{ThmEasyGlob} and Theorem \ref{ThmEasyLoca}  obey the following uniform scheme, as now being sketched in a general version. Let $\CC\subset NC^{\twocol}$ be a category of noncrossing partitions and let $G$ be the associated free easy quantum group with $C^*$-algebra $C(G)$. Let $k$ and $d$ be the parameters of the category appearing in Theorems \ref{ThmEasyGlob} and \ref{ThmEasyLoca} (recall also Section \ref{SectCases}). Consider the following diagram:

\begin{align*}
M_t(C(G)) && &&\stackrel{\beta}{\longleftarrow}&& && A\square C^*(\Z_k)\\
&&v_{ij}',s', z' &&\mapsfrom && v_{ij}, s, z&&\\
&\\
\uparrow\iota && && && &&\upsubset\\
&\\
&& u_{ij} &&\mapsto &&u_{ij}':=v_{ij}sz\\
C(G) && &&\stackrel{\alpha}{\longrightarrow} && &&C^*(v_{ij}sz)
\end{align*}

In the above diagram, $\square\in\{*,\otimes_{\max}\}$, and the $C^*$-algebra $A$ is specified case by case. Typically, it is of the form $A=C(H)$ or $A=C(H)*C^*(\Z_d)$ where $C(H)$ is generated by elements $v_{ij}$ and $C^*(\Z_d)$ by $s$ (or $s=1$, if $A=C(H)$). The $C^*$-algebra $C^*(\Z_k)$ in turn is generated by $z$. Finally, we will have $t=d$ in most of the cases.

The general strategy is now:
\begin{itemize}
\item Using the universal property, we prove that $\alpha$ exists, mapping $u_{ij}\mapsto v_{ij}sz$. Here, we may usually trace back a relation $R(p)$ of $C(G)$ to exactly the same relation $R(p)$ of $C(H)$. It is clear that $\alpha$ is surjective.
\item As for finding a converse to $\alpha$, we would like to use the universal property of $A\square C^*(\Z_k)$ (see Remark \ref{RemUniv}). Unfortunately, we may not find all elements of this $C^*$-algebra in $C(G)$ itself. For instance, we might find a candidate for $z^d$ in $C(G)$ (implicitely using the fact that $(sz)^d=z^d$), but not $z$ itself. We thus somehow have to find a $d$-th root of $z^d$, which we do by passing to $M_d(C(G))$. We prove that $\beta$ exists with $\beta(v_{ij})=\iota(u_{ij})\beta(sz)^*$.
\item The map $\iota$ is supposed to be injective. If now the diagram commutes, then $\alpha$ is injective and we have proven the desired isomorphism. Hence we identified the two quantum groups in the sense of Definition \ref{DefW}(a).
\end{itemize}

\begin{proof}[Proof of Theorem \ref{ThmEasyGlob}]
Let $\CC(k)$ be one of the globally colorized categories, with $k\in\N_0$ and let $G$ be the associated quantum group. We use the uniform proof scheme as sketched above with the  parameters $s:=1$, $\square:=\otimes_{\max}$, and $t$ and $A$ as below. For the convenience of the reader, we recall the generators of the globally colorized categories from Theorem \ref{ThmClassiGlob} and  we also list all partitions that generate the category of partitions associated to $A$ (from Remark \ref{Reminder}).

\quad

\begin{tabular}{|l|l|l|l|l|l|}
\hline
$\CC(k)$ & $=\langle\paarpartww\otimes\paarpartbb  \textnormal{ \& those below}\rangle$ & $A$ &$\leftrightarrow\langle\idpartwb \textnormal{ \& those below}\rangle$&$t$\\
\hline
 $\categ{\OOO}{\glob}{k}$ & $\paarpartww^{\nest(\frac{k}{2})}$ & $C(O_n^+)$ &$\emptyset$ & 2\\
 $\categ{\HHH}{\glob}{k}$ & $b_k,\vierpartwbwb$ & $C(H_n^+)$ &$\vierpartwbwb$ & 2\\
 $\categ{\SSS}{\glob}{k}$ & $\singletonw^{\otimes k},\vierpartwbwb,\singletonw\otimes\singletonb$ & $C(S_n^+)$ &$\vierpartwbwb,\singletonw$ & 1\\
  $\categ{\BBB}{\glob}{k}$ & $\singletonw^{\otimes k},\singletonw\otimes\singletonb$ & $C(B_n^{\#+})$ &$\singletonw\otimes\singletonb$ &2\\
 $\categ{\BBB'}{\glob}{k}$ & $\singletonw^{\otimes k},\positionerwwbb,\singletonw\otimes\singletonb$ & $C(B_n^+)$ &$\singletonw$ & 1\\
\hline
\end{tabular}

\quad

Thus, $t=d(\CC)$ in the sense of \cite[Sect. 3.1, 4.1, 5.1, 6.1]{tarragowebercombina}. Recall that the restriction on the parameter $k$ in Theorem \ref{ThmEasyGlob} is exactly that $k$ is even in case $t=2$.

\emph{Step 1: The map $\alpha:C(G)\to A\otimes_{\max}C^*(\Z_k)$ exists.} 
It is straightforward to check that the elements $u_{ij}':=v_{ij}z\in A\otimes_{\max}C^*(\Z_k)$ give rise to unitary matrices $u':=(u_{ij}')$ and $\bar u':=(u_{ij}'^*)$ and that the relations $R(\paarpartww\otimes\paarpartbb)$ are fulfilled (recall that all necessary $C^*$-algebraic relations may be found in Section \ref{SectCStar}). Moreover, the relations $R(\paarpartww^{\nest(\frac{k}{2})})$ are fulfilled in the case $\CC(k)=\categ{\OOO}{\glob}{k}$. Here, we used that $z^\lambda=(z^*)^\lambda$ for $k=2\lambda$. Next, we check that the $u_{ij}'$ fulfill $R(\vierpartwbwb)$ and $R(b_k)$, if the $v_{ij}$ fulfill $R(\vierpartwbwb)$ (recall that the $v_{ij}$ are selfadjoint). Using the fact that the $v_{ij}^2$ are orthogonal projections adding up to 1 when the sum is taken over $j$, this is a direct computation. Similarly, we prove that $R(\singletonw\otimes\singletonb)$ and $R(\singletonw^{\otimes k})$ are fulfilled by the $u_{ij}'$, if $R(\singletonw\otimes\singletonb)$ is fulfilled by the $v_{ij}$ or maybe even $R(\singletonw)$. For doing so, use:
\[\left(\sum_{l_1}v_{l_1j_1}\right)\left(\sum_{l_2}v_{l_2j_2}\right)=\sum_{l_1}\left(v_{l_1j_1}\left(\sum_{l_2}v_{l_1l_2}\right)
\right)=\sum_{l_2}\sum_{l_1}v_{l_1j_1}v_{l_1l_2}=\sum_{l_2} \delta_{j_1l_2}=1\]
Finally, $R(\positionerwwbb)$ is fulfilled by the $u_{ij}'$ in case the $v_{ij}$ satisfy $R(\singletonw)$.

\emph{Step 2: The map $\beta: A\otimes_{\max}C^*(\Z_k)\to M_t(C(G))$ exists in the case $t=1$.} 
We put $z':=\sum_l u_{il}=\sum_l u_{lj}\in C(G)$ and check that it is independent of the choice of $i$ and $j$, since $\singletonw\otimes\singletonb\in\CC(k)$. Moreover it is unitary, since: 
\[\sum_l \left(u_{il}^*\left(\sum_m u_{ml}\right)\right)=\sum_m\sum_l u_{il}^*u_{ml}=\sum_m\delta_{im}=1\]
The relation $z'^k=1$ immediately follows from $\singletonw^{\otimes k}\in\CC(k)$.
We put $v_{ij}':=u_{ij}z'^*$ and check that it is selfadjoint since firstly, using $R(\paarpartww\otimes\paarpartbb)$ we have
\[u_{ij}z'^*=\sum_lu_{ij}u_{kl}^*=\sum_lu_{ij}^*u_{kl}=u_{ij}^*z'\]
and secondly we have $u_{ij}z'=z'u_{ij}$ because $\positionerwwbb\in\CC(k)$. This also proves $v_{ij}'z'=z'v_{ij}'=u_{ij}$.
Finally, check that $v_{ij}'$ satisfies $R(\singletonw)$ and $R(\vierpartwbwb)$ (the latter one in the case $\CC(k)=\categ{\SSS}{\glob}{k}$) from which we infer the existence of $\beta$. Note that we used $\singletonw^{\otimes k}$ only for proving $z'^k=1$, so the proof works also in the case $k=0$.

\emph{Step 3: The map $\beta: A\otimes_{\max}C^*(\Z_k)\to M_t(C(G))$ exists in the case $t=2$.} 
We cannot work with the same $z'$ as in Step 2, since this element will not commute with $v_{ij}'$ because $\positionerwwbb\notin\CC(k)$. We therefore consider the auxiliary element
$w:=\sum_l u_{il}^2=\sum_l u_{lj}^2\in C(G)$ and check that it is independent of the choice of $i$ or $j$ using $R(\paarpartww\otimes\paarpartbb)$ and:
\[\sum_{l}u_{il}^2=\sum_{l,m} u_{il}u_{il}u_{mj}^*u_{mj}=\sum_{l,m} u_{il}^*u_{il}u_{mj}u_{mj}=\sum_m u_{mj}^2\]
Similarly, we see that $w$ is unitary. Moreover, since $k$ is even in the case $t=2$, we may write $k=2\lambda$. Note that $\paarpartww^{\nest(\lambda)}$ is always in $\CC(k)$, since the permutation of colors  allows us to infer $\paarpartww^{\nest(\lambda)}\otimes \pi_{2\lambda}'\in\CC(k)$ from 
 $\paarpartwb^{\nest(\lambda)}\otimes \pi_{2\lambda}\in\CC(k)$, where $\pi_{2\lambda}$ is any partition in $P^{\twocol}(0,2\lambda)$ consisting only of white points and $\pi_{2\lambda}'\in P^{\twocol}(0,2\lambda)$ consists of $\lambda$ white points and $\lambda$ black points; now, by  \cite[Lemma 1.1(b)]{tarragowebercombina}, we have $\paarpartww^{\nest(\lambda)}\in\CC(k)$. Using $R(\paarpartww^{\nest(\lambda)})$, we may replace $\lambda$ many occurences of the $u_{ij}$ by $u_{ij}^*$, hence:
\[w^\lambda=\sum_{l_1,\ldots,l_\lambda}u_{i_1l_1}u_{i_1l_1}^*u_{i_2l_2}u_{i_2l_2}^*\ldots u_{i_\lambda l_\lambda}u_{i_\lambda l_\lambda}^*=1\]
Using $R(\paarpartww\otimes\paarpartbb)$, it is easy to see that $u_{ij}^*w=u_{ij}=wu_{ij}^*$.
We now define:
\[z':=\begin{pmatrix} 0 & w\\ 1 & 0\end{pmatrix}\in M_2(C(G)),\qquad v_{ij}':=\begin{pmatrix} 0 & u_{ij}\\ u_{ij}^* & 0\end{pmatrix}\in M_2(C(G))\]
It is clear that $z'$ is unitary with $z'^k=1$, the $v_{ij}'$ are selfadjoint and fulfill the relations $R(p)$ of $A$ whenever $p\in\CC(k)$. Finally, $v_{ij}'z'=z'v_{ij}'$ and hence $\beta$ exists. Again, this step also works in the case $k=0$.

\emph{Step 4: With $\iota: C(G)\to M_t(\C)\otimes C(G), x\mapsto 1\otimes x$ the diagram commutes.} 
Check that $v_{ij}'z'=\iota(u_{ij})$ in the cases $t=1$ and $t=2$. We conclude that $\alpha$ is an isomorphism.
\end{proof}

\begin{proof}[Proof of Theorem \ref{ThmEasyLoca}] \quad

\emph{Case 1.}
In the cases $\CC\in\{\categg{\OOO}{\loc},\categg{\HHH'}{\loc}\}$, we use the uniform scheme with $t=2$, $A\in \{C(O_n^+),C(H_n^+)\}$, $\square=*$, $k=0$ (i.e. $\Z_k=\Z$) and $s=1$. Putting
\[z':=\begin{pmatrix} 0&1\\1&0\end{pmatrix},\qquad v_{ij}':=\begin{pmatrix} 0&u_{ij}\\u_{ij}^*&0\end{pmatrix},\qquad \iota(u_{ij}):=\begin{pmatrix}u_{ij} &0\\0&u_{ij}^*\end{pmatrix}\]
we immediately see that $\alpha$ and $\beta$ exist and that the diagram is commutative, proving that $\alpha$ is an isomorphism. In fact, we could also have chosen $k=2$, thus we just proved $H_n^+\freeglued \Z\cong H_n^+\freeglued \Z_2$ and $O_n^+\freeglued \Z\cong O_n^+\freeglued \Z_2$, as a byproduct (see also Proposition \ref{OnIsom}).

\emph{Case 2.}
As for the other cases of $\CC$ we choose the following parameters, if $d\neq 0$.

\quad

\begin{tabular}{|l|l|l|}
\hline
$\CC(k,d)$ & $=\langle\ldots\rangle$ & $A$ \\
\hline
 $\categ{\SSS}{\loc}{k,d}$ & $\singletonw^{\otimes k}, \positionerd,\vierpartwbwb,\singletonw\otimes\singletonb$ & $C(S_n^+)*C^*(\Z_d)$\\
 $\categ{\BBB}{\loc}{k,d}$ & $\singletonw^{\otimes k}, \positionerd,\singletonw\otimes\singletonb$ & $C(C_n^+)*C^*(\Z_d)$\\
 $\categ{\BBB'}{\loc}{k,d,0}$ & $\singletonw^{\otimes k}, \positionerd,\positionerwbwb,\singletonw\otimes\singletonb$ & $C(B_n^+)*C^*(\Z_d)$\\
  $\categ{\BBB'}{\loc}{k,d,r}$ & $\singletonw^{\otimes k}, \positionerd,\positionerrpluseins,\singletonw\otimes\singletonb$ & $C(C_n^+)*C^*(\Z_d)/\langle v_{ij}s^r=s^rv_{ij}\rangle$\\
   $\categ{\HHH}{\loc}{k,d}$ & $b_k,b_d\otimes \tilde b_d,\vierpartwwbb,\vierpartwbwb$ & $C(\widehat{\Z_d}\wr_* S_n^+)$\\ 
 \hline
\end{tabular}

\quad

Furthermore, we put $t:=d$, $\square:=\otimes_{\max}$ and we let $z$ be the generator of $C^*(\Z_k)$ whereas $s$ is the generator of $C^*(\Z_d)$. In the case $\categ{\HHH}{\loc}{k,d}$, we refine the scheme replacing $s$ by the generators $s^{(i)}$ of the copies of $C^*(\Z_d)$ in $C(\widehat{\Z_d}\wr_* S_n^+)$. Let us always write $s^{(i)}$ in the sequel, meaning $s^{(i)}=s$ if $\CC(k,d)\neq \categ{\HHH}{\loc}{k,d}$.

The existence of $\alpha:C(G)\to A\otimes_{max} C^*(\Z_k)$ mapping $u_{ij}\mapsto u_{ij}':=v_{ij}s^{(i)}z$ is a straightforward calculation.

As for $\beta$, we proceed as follows.
If $\CC(k,d)\neq \categ{\HHH}{\loc}{k,d}$ we put $w:=\sum_l u_{il}=\sum_l u_{lj}$ and we use $\singletonw\otimes\singletonb\in\CC(k,d)$ in order to prove that this definition is independent of $i$ and $j$. Moreover, an easy computation shows that $w$ is unitary with $w^k=1$. If $\CC(k,d)= \categ{\HHH}{\loc}{k,d}$ the elements $w_i:=\sum_l u_{il}$ do depend on $i$, but they are still unitary with $w_i^k=1$. Using $R(\vierpartwwbb)$ and $R(b_d\otimes\tilde b_d)$, we may show that $w_i^d=\sum_lu_{il}^d=\sum_l u_{lj}^d$, hence the $d$-th powers of the $w_i$ coincide. We then put, in all five cases of $\CC(k,d)$
\[z':=\begin{pmatrix} 0&1& &\\&\ddots&\ddots&\\&&\ddots&1\\w_i^d&&&0\end{pmatrix},\qquad
s_i':=\begin{pmatrix} 0&& &(w_i^*)^{d-1}\\w_i&\ddots&&\\&\ddots&\ddots&\\&&w_i&0\end{pmatrix}\]
and $\iota(u_{ij}):=1\otimes u_{ij}$, $v_{ij}':=\iota(u_{ij}w_i^*)$ with  $w_i=w$ and $s'=s_i'$ in the case  $\CC(k,d)\neq \categ{\HHH}{\loc}{k,d}$. We check that $z'$ and $s_i'$ are unitaries with $z'^k=s_i'^d=1$. Recall that $d$ is a divisor of $k$. Moreover, we have $z'^d=\iota(w_i^d)$, $\iota(w_i)=z's_i'=s_i'z'$ and $u_{ij}w_i^d=w_i^du_{ij}$ in all cases. Additionally, we have $u_{ij}w_i^*=w_iu_{ij}^*$, if $\vierpartwbwb$ or $\positionerwbwb$ is in $\CC(k,d)$. Finally, $u_{ij}w_i=w_iu_{ij}$, if $\vierpartwwbb\in\CC(k,d)$.
The existence of $\beta$ with $\beta\circ \alpha=\iota$ is now a routine check.

\emph{Case 3.} Now let $\CC$ be as in Case 2 but with $d=k$ (including $d=0$). It is clear from the combinatorics that $\CC(k,k)=\CC(k,0)$, see Theorem \ref{ThmClassiLoc}.
If $d=0$, we use $t=1$ and  the following table for the choice of $B:=A\square C^*(\Z_k)$:

\quad

\begin{tabular}{|l|l|}
\hline
$\CC(k,0)$ &$B:=A\square C^*(\Z_k)$ \\
\hline
 $\categ{\SSS}{\loc}{k,0}$ & $C(S_n^+)*C^*(\Z_k)$\\
 $\categ{\BBB}{\loc}{k,0}$ &  $C(C_n^+)*C^*(\Z_k)$\\
 $\categ{\BBB'}{\loc}{k,0,0}$ &  $C(B_n^+)*C^*(\Z_k)$\\
 $\categ{\BBB'}{\loc}{k,0,\frac{k}{2}}$ &  $C(C_n^+)*C^*(\Z_k)/\langle v_{ij}s^r=s^rv_{ij}\rangle$\\
 $\categ{\HHH}{\loc}{k,0}$ &  $C(\widehat{\Z_k}\wr_* S_n^+)$\\ 
 \hline
\end{tabular}

\quad

The existence of $\alpha$ is as before, while $\beta: B\to C(G)$ maps the generator $s$ of $C^*(\Z_k)$ (or rather $s^{(i)}$) to $\sum_l u_{il}$ and $v_{ij}$ to $u_{ij}\left(\sum_l u_{il}^*\right)$.
\end{proof}

 We now take up the statements of Remarks \ref{RemConfusionWithProducts} and \ref{RemXYZ} and the one at the end of Case 1 of the proof of Theorem \ref{ThmEasyLoca}.
We prove them using again the scheme of the proofs of Theorems \ref{ThmEasyGlob} and \ref{ThmEasyLoca}.

\begin{prop}\label{OnIsom}
We have the following isomorphisms (in the sense of Def. \ref{DefW}(a)):
\begin{itemize}
\item[(a)] $O_n^+\cong O_n^+\tensorglued\widehat{\Z_2}$, $H_n^+\cong H_n^+\tensorglued\widehat{\Z_2}$  and $B_n^{\#+}\cong B_n^{\#+}\tensorglued\widehat{\Z_2}$.
\item[(b)] $H_n^+\freeglued \widehat{\Z}\cong H_n^+\freeglued\widehat{\Z_2}$ and $O_n^+\freeglued \widehat{\Z}\cong O_n^+\freeglued\widehat{\Z_2}$.
\item[(c)] $(\widehat{\Z_k}\wr_*S_n^+)\tensorglued \widehat{\Z_k}\cong \widehat{\Z_k}\wr_*S_n^+$, $(C_n^+\freeglued_r \widehat{\Z_k}) \tensorglued \widehat{\Z_k}\cong C_n^+\freeglued_r \widehat{\Z_k}$ and $(G\freeglued \widehat{\Z_k}) \tensorglued \widehat{\Z_k}\cong G\freeglued \widehat{\Z_k}$ for $G\in\{S_n^+, C_n^+, B_n^+\}$.
\end{itemize}
\end{prop}
\begin{proof}
(a) Using the proof scheme presented at the beginning of this subsection with the parameters $G\in\{O_n^+,H_n^+, B_n^{\#+}\}$, $A\square C^*(\Z_k)=C(G)\otimes_{\max} C^*(\Z_2)$, $t=2$, \[v_{ij}'=\begin{pmatrix}0&u_{ij}\\u_{ij}&0\end{pmatrix},\qquad,z':=\begin{pmatrix}0&1\\1&0\end{pmatrix}, \qquad \iota(u_{ij})=\begin{pmatrix}u_{ij}&0\\0&u_{ij}\end{pmatrix}\]
we obtain the result. For the existence of $\alpha$ it is crucial that the $v_{ij}z$ are selfadjoint, hence we need to have $z=z^*$. Thus, the proof does not carry over for general $\Z_k$. Furthermore, we have $\sum_j v_{ij}z=z\neq 1$, so the proof does not work for $G=S_n^+$ or $G=B_n^+$.

(b) Check that for $G\in\{O_n^+,H_n^+\}$ the maps
\begin{align*}
\alpha: C(G)*C^*(\Z)&\to C(G)*C^*(\Z_2)\\
u_{ij}&\mapsto u_{ij}\\
z&\mapsto z
\end{align*}
\begin{align*}
\beta:C(G)*C^*(\Z_2) &\to M_2(C(G)*C^*(\Z)) \\
 u_{ij}&\mapsto \begin{pmatrix}0&u_{ij}\\u_{ij}&0\end{pmatrix}\\
z&\mapsto \begin{pmatrix}0&z^*\\z&0\end{pmatrix}
\end{align*}
\begin{align*}
\iota: C(G)*C^*(\Z) &\to M_2(C(G)*C^*(\Z))\\
u_{ij}&\mapsto \begin{pmatrix}u_{ij}&0\\0&u_{ij}\end{pmatrix}\\
z&\mapsto \begin{pmatrix}z&0\\0&z^* \end{pmatrix}
\end{align*}
exist, with $\iota$ being injective and $\iota(u_{ij}z)=\beta(\alpha(u_{ij}z))$.

(c) Again we use a version of the general proof scheme showing that the following maps exist for $A=C(G)*C^*(\Z_k)$, $A=C(C_n^+)*C^*(\Z_k)/\langle u_{ij}s^r=s^ru_{ij}\rangle$ or $A=C^*(\Z_k)^{*n}*C(S_n^+)/\langle u_{ij}s^{(i)}=s^{(i)}u_{ij}\rangle$.
\begin{align*}
\alpha:A&\to A\otimes C^*(\Z_k)\\
u_{ij}&\mapsto u_{ij}\\
s^{(i)}&\mapsto s^{(i)}z
\end{align*}
\begin{align*}
\beta:A\otimes C^*(\Z_k)&\to M_k(A) \\
 u_{ij}&\mapsto \iota(u_{ij})\\
s^{(i)}&\mapsto\begin{pmatrix} 0&& &s^{(i)}\\s^{(i)}&\ddots&&\\&\ddots&\ddots&\\&&s^{(i)}&0\end{pmatrix}\\
z&\mapsto\begin{pmatrix} 0&1& &\\&\ddots&\ddots&\\&&\ddots&1\\1&&&0\end{pmatrix}
\end{align*}
\begin{align*}
\iota: A &\to M_k(A)\\
x&\mapsto 1\otimes x
\end{align*}
\end{proof}

\subsection{The quantum reflection groups $H_n^{s+}$}\label{SectRefl}

The quantum reflection groups $H_n^{s+}$ were defined by Banica, Belinschi, Capitaine and Collins in \cite[Sect. 11]{banica2011free}.

\begin{defn}\label{DefReflectionQG}
Given $n,s\in\N$, the \emph{quantum reflection group} $H_n^{s+}$ is given by the universal $C^*$-algebra generated by elements $u_{ij}$, $1\leq i,j\leq n$ subject to the conditions:
\begin{itemize}
\item $u=(u_{ij})$ and $\bar u=(u_{ij}^*)$ are unitaries
\item all $u_{ij}$ are partial isometries (i.e. $u_{ij}u_{ij}^*u_{ij}=u_{ij}$)  and the projections $u_{ij}^*u_{ij}$ and $u_{ij}u_{ij}^*$ coincide
\item $u_{ij}^s=u_{ij}u_{ij}^*$
\end{itemize}
We define $H_n^{\infty +}$ by omitting the third of the above conditions.
\end{defn}

We have  $H_n^{1+}=S_n^+$, $H_n^{2+}=H_n^+$ and  $H_n^{s+}=\widehat{\Z_s} \wr_* S_n^+$ by \cite[Thm. 3.4]{banica2009fusion}. Hence, they are free easy, by Theorem \ref{ThmEasyLoca}. As an alternative proof, check that the categories found in \cite{banica2009fusion} in order to describe the fusion rules of $H_n^{s+}$ are exactly $\categ{\HHH}{\loc}{s,0}=\langle b_s,\vierpartwwbb,\vierpartwbwb\rangle$. 

As a third proof, check that the universal $C^*$-algebra $A$ of Definition \ref{DefReflectionQG} is isomorphic to the $C^*$-algebra $B$ generated by elements $v_{ij}$ and the relations $R(b_s)$, $R(\vierpartwwbb)$, $R(\vierpartwbwb)$ and those making $v$ and $\bar v$ unitaries, mapping generators to generators. 
Indeed, fixing $k$, the projections $u_{ik}u_{ik}^*\in A$ are orthogonal since they sum up to one. Hence they fulfill $R(\vierpartwbwb)$ and also $R(\vierpartwwbb)$, because for $i\neq j$:
\[u_{ik}u_{jk}=u_{ik}u_{ik}^*u_{ik}u_{jk}u_{jk}^*u_{jk}=u_{ik}u_{ik}u_{ik}^*u_{jk}u_{jk}^*u_{jk}=0\]
Moreover we have $R(b_s)$, since
\[\sum_l u_{i_1l}\ldots u_{i_sl}=\delta_{i_1i_2}\delta_{i_2i_3}\ldots \delta_{i_{s-1}i_s}\sum_l u_{i_1l}^s\]
as a consequence of $R(\vierpartwwbb)$. Now, $\sum_l u_{i_1l}^s=\sum_l u_{i_1l}u_{i_1l}^*=1$.
Conversely, $v_{ij}^{s-1}=v_{ij}^*$ in $B$, as a consequence of $R(\rot_1(b_s))$. Hence, $v_{ij}^*v_{ij}=v_{ij}^s=v_{ij}v_{ij}^*$. Furthermore, use $R(\vierpartwbwb)$ to show:
\[v_{ij}^*v_{ij}=\sum_k v_{ij}^*v_{ik}v_{ik}^*v_{ij}=v_{ij}^*v_{ij}v_{ij}^*v_{ij}\]

\section{Unitary easy groups}\label{SectUnitaryGroups}

Treating the case of groups in the setting of easy quantum groups brings us back to the very beginning of Woronowicz's Tannaka-Krein machine for compact matrix quantum groups: We now recall Schur-Weyl duality. This section can be read as an introduction to easy quantum groups by someone not familiar with quantum groups or operator algebras.

Let $U_n\subset M_n(\C)$ be the group of unitary matrices. It acts naturally on the Hilbert space $(\C^n)^{\otimes k}$ by $\pi:U_n\to B((\C^n)^{\otimes k})$ defined by $\pi(u):=u^{\otimes k}$, i.e.: \[\pi(u)(e_{i_1}\otimes \ldots\otimes e_{i_k})=
ue_{i_1}\otimes \ldots\otimes ue_{i_k}=
\sum_{j_1,\ldots,j_k} (u_{j_1i_1}\ldots u_{j_ki_k})(e_{j_1}\otimes\ldots\otimes e_{j_k})\]

The permutation group $S_k$ in turn acts on $(\C^n)^{\otimes k}$ by $\rho:S_k\to B((\C^n)^{\otimes k})$  given by:
\[\rho(\sigma)(e_{i_1}\otimes \ldots\otimes e_{i_k})=
e_{i_{\sigma(1)}}\otimes \ldots\otimes e_{i_{\sigma(k)}}\]

Schur-Weyl duality now states that the commutant of $\spanlin \pi(U_n)$ is $\spanlin \rho(S_k)$ and vice versa. In other words:
\[\Mor(u^{\otimes k},u^{\otimes k}):=\{T:(\C^n)^{\otimes k}\to (\C^n)^{\otimes k} \textnormal{ linear }\;|\; Tu^{\otimes k}=u^{\otimes k}T\}=\spanlin \rho(S_k)\]

Rephrasing this fact in the language of this article, we first observe that to each permutation $\sigma\in S_k$, we may associate a pair partition $p_{\sigma}\in P_2^{\twocol}(k,k)$ connecting the $i$-th upper point with the $j$-th lower point if and only if $\sigma(i)=j$. All points of $p_{\sigma}$ are white. 

\begin{ex}\quad

\setlength{\unitlength}{0.5cm}
\begin{center}
\begin{picture}(16,2)
\put(0,1){$\sigma=\begin{pmatrix} 1&2&3&4&5\\2&4&1&3&5\end{pmatrix}$}
\put(10,1){$p_{\sigma}=$}
\put(12,0){$\circ$}
\put(13,0){$\circ$}
\put(14,0){$\circ$}
\put(15,0){$\circ$}
\put(16,0){$\circ$}
\put(12,2.4){$\circ$}
\put(13,2.4){$\circ$}
\put(14,2.4){$\circ$}
\put(15,2.4){$\circ$}
\put(16,2.4){$\circ$}
\put(12.3,0.4){\line(1,1){2}}
\put(12.3,2.4){\line(1,-2){1}}
\put(13.3,2.4){\line(1,-1){2}}
\put(14.3,0.4){\line(1,2){1}}
\put(16.2,0.4){\line(0,1){2}}
\end{picture}
\end{center}
\end{ex}

Now, let $\CC=\langle\crosspartwwww\rangle$ be the category of partitions generated by $\crosspartwwww$ and the base partitions $\paarpartwb$ and $\idpartww$ (see Section \ref{SectCateg}). We then have:
\[\{p\in\CC(k,k) \textnormal{ all points are white}\}=\{p_{\sigma}\in P_s^{\twocol}(k,k)\;|\; \sigma\in S_k\}\]

Moreover, we have $\rho(\sigma)=T_{p_{\sigma^{-1}}}$ with the definition of $T_{p_{\sigma^{-1}}}$ as in Section \ref{SectUnitaryEasy2} and thus:
\[\Mor(u^{\otimes k},u^{\otimes k})=\spanlin\{T_p\;|\; p\in\CC(k,k)\textnormal{ all points are white}\}\]

This proves that the intertwiner spaces $\Mor(u^{\otimes k},u^{\otimes k})$ of $U_n$ are given by $\CC(k,k)$, in the sense of easy quantum groups. The question is now: Let $G\subset U_n\subset M_n(\C)$ be a subgroup and let it act on $(\C^n)^{\otimes k}$ via $\pi(u):=u^{\otimes k}$. How can $\Mor(u^{\otimes k},u^{\otimes k})$ be described, how can we write
\[\Mor(u^{\otimes k},u^{\otimes k})=\spanlin\{T_p\;|\; p\in\CC(k,k)\textnormal{ all points are white}\}\]
for a suitable set $\CC(k,k)$? This is exactly the starting point for Tannaka-Krein and its quantum version leading to easy quantum groups.

Let us now describe all unitary easy groups.
In \cite{tarragowebercombina}, the categories \emph{in the group case}, i.e. those containing the \emph{crossing partition} $\crosspartwwww\in P^{\twocol}(2,2)$ are classified. (Equivalently, $\crosspartwbbw$ etc. is in the category.) They are listed in the following theorem.

\begin{thm}[Thm 8.3 \cite{tarragowebercombina}]
The categories in the group case are the following.
\begin{itemize}
\item $\categ{\OOO}{\grp,\glob}{k}=\langle\paarpartww^{\otimes \nest(\frac{k}{2})},\paarpartww\otimes\paarpartbb,\crosspartwwww\rangle$ for $k\in 2\N_0$
\item $\categg{\OOO}{\grp,\loc}=\langle\crosspartwwww\rangle$
\item $\categ{\HHH}{\grp,\glob}{k}=\langle b_k,\vierpartwbwb,\paarpartww\otimes\paarpartbb,\crosspartwwww\rangle$ for $k\in 2\N_0$
\item $\categ{\HHH}{\grp,\loc}{k,d}=\langle b_k,b_d\otimes\tilde b_d,\vierpartwbwb,\crosspartwwww\rangle$ for $k,d\in\N_0\backslash\{1,2\}$, $d\vert k$
\item $\categ{\SSS}{\grp,\glob}{k}=\langle \singletonw^{\otimes k},\vierpartwbwb,\singletonw\otimes\singletonb,\paarpartww\otimes\paarpartbb,\crosspartwwww\rangle$ for $k\in \N_0$
\item $\categ{\BBB}{\grp,\glob}{k}=\langle\singletonw^{\otimes k}, \singletonw\otimes\singletonb,\paarpartww\otimes\paarpartbb,\crosspartwwww\rangle$ for $k\in 2\N_0$
\item $\categ{\BBB}{\grp,\loc}{k}=\langle \singletonw^{\otimes k}, \singletonw\otimes\singletonb,\crosspartwwww\rangle$ for $ k\in \N_0$ 
\end{itemize}
\end{thm}

If $\CC$ is a category containing the crossing partitions $\crosspartwwww$ and $\crosspartwbbw$, then the $C^*$-algebra associated to it is commutative (see the relations in Section \ref{SectCStar}). Hence, the associated quantum groups are in fact groups. They are listed in the next theorem. Let us prepare the statement.

If $G\subset M_n(\C)$ is a group, we denote by $G\tensorglued\Z_k$ the subgroup of $G\times \Z_k=\{(g,x)\;|\; g\in G, x\in \Z_k\}$ generated by $\{(g,z)\;|\;g\in G\}$ where $z$ is the generator of $\Z_k=\Z/k\Z$. Moreover, let
\begin{itemize}
\item $O_n\subset M_n(\C)$ be the group of orthogonal matrices,
\item $\Z_k\wr S_n=(\Z_k^{\times n})\rtimes S_n$ denote the wreath product of $\Z_k$ with $S_n$, with $H_n:=\Z_2\wr S_n$ the Coxeter group of type B,
\item $B_n\subset M_n(\C)$ be the bistochastic group consisting of all orthogonal matrices summing up to one in each row and each column,
\item $C_n\subset M_n(\C)$ be the group of all unitary matrices summing up to one in each row and each column.
\end{itemize}
See also \cite[Prop. 2.4]{banica2009liberation}.

\begin{thm}\label{ThmGroup}
The groups corresponding to the categories in the group case are the following:
\begin{itemize}
\item $\categ{\OOO}{\grp,\glob}{k}:O_n\tensorglued \Z_k$, for $k\in 2\N_0$
\item $\categg{\OOO}{\grp,\loc}:U_{n}$
\item $\categ{\HHH}{\grp,\glob}{k}:(\Z_2\wr S_n)\tensorglued \Z_k=H_n\tensorglued \Z_k$, for $k\in 2\N_0$
\item $\categ{\HHH}{\grp,\loc}{k,d}:(\Z_d\wr S_n)\tensorglued \Z_k$, for $k,d\in \N_0\backslash\{1,2\}$, $k$ a multiple of $d$
\item $\categ{\SSS}{\grp,\glob}{k}:S_n\tensorglued \Z_k$, for $k\in \N_0$
\item $\categ{\BBB}{\grp,\glob}{k}:B_n\tensorglued\Z_k$, for $k\in 2\N_0$
\item $\categ{\BBB}{\grp,\loc}{k}:C_n\tensorglued \Z_k$, for $k\in \N_0$
\end{itemize}
\end{thm}
\begin{proof}
Let $\CC$ be a category in the group case and let $G$ be the associated quantum group with $C^*$-algebra $C(G)$. Since $\crosspartwwww\in\CC$ and $\crosspartwbbw\in\CC$, the $C^*$-algebra $C(G)$ is  commutative, so $G$ is actually a group.

If $\CC=\categg{\OOO}{\grp,\loc}$, it is immediately clear that $C(G)=C(U_n)$ and hence $G=U_n$.
In the other cases, all we have to do is to read the proofs of Theorems \ref{ThmEasyGlob} and \ref{ThmEasyLoca} again and carefully check that the maps $\alpha$ and $\beta$ still exist, i.e. that the elements $u_{ij}'$ commute and likewise for the $v_{ij}'$. Note that if $\CC$ is one of the categories $\categ{\OOO}{\grp,\glob}{k}$, $\categ{\HHH}{\grp,\glob}{k}$, $\categ{\SSS}{\grp,\glob}{k}$ or  $\categ{\BBB}{\grp,\glob}{k}$, we consider the groups $H\in\{O_n,H_n,S_n,B_n\}$. We may view $C(H\tensorglued \Z_k)$ as the $C^*$-subalgebra $C^*(v_{ij}z)$ of $C(H)\otimes C^*(\Z_k)$. The cases of $\categ{\HHH}{\grp,\loc}{k,d}$ and $\categ{\BBB}{\grp,\loc}{k}$ are similar.
\end{proof}

\begin{rem}
As an alternative proof, let $\CC$ be a category in the group case and let $G\subset M_n(\C)$ be the group consisting of all matrices $u=(u_{ij})$ satisfying the relations $R(p)$ of Section 2, for $p$ being the generators of $\CC$. One can check that we have a surjective group homomorphism from $G$ onto the corresponding group of the statement in Theorem \ref{ThmGroup}. It remains to show that the homomorphism is injective. In our proof, we do it by showing that the function algebras over these groups are isomorphic, which by the general theory of $C^*$-algebras shows that the underlying topological spaces are homeomorphic. But there is certainly also a direct proof within group theory.
\end{rem}

\section{Concluding remarks}\label{SectConcluding}

\subsection{Open problems in the classification of unitary easy quantum groups}

The classification of all categories $\langle\emptyset\rangle\subset\CC\subset NC^{\twocol}$ obtained in \cite{tarragowebercombina} amounts to knowing all easy quantum groups  $S_n^+\subset G\subset U_n^+$. The classification of all categories containing the crossing partitions $\crosspartwwww$, $\crosspartwbbw$, etc. yields all easy groups $S_n\subset G\subset U_n$. Moreover, as an outcome of the classification in the orthogonal case we know all easy quantum groups $S_n\subset G\subset S_n^+$  -- there are none besides $S_n$ and $S_n^+$, see \cite{raum2013full} or \cite{banica2009liberation}. However, so far we have no result on determining the easy quantum groups $U_n\subset G\subset U_n^+$, which would be a natural next step to do in order to complete the classification of all four sides of the following square:
\begin{align*}
S_n^+ &&\subset && U_n^+\\
\subsetup && &&\subsetup\\
S_n &&\subset && U_n
\end{align*}
In the long run,  the diagonal of that square needs to be classified, too, hence all  easy quantum groups $S_n\subset G\subset U_n^+$ need to be found.
Like in the orthogonal case, one could start with  the half-liberated case, thus with categories containing some half-liberated partitions $\halflibpartwwwwww$ with various colorings, not containing the crossing partitions $\crosspartwwww$ etc. This should immediately give at least one example of a quantum group $U_n\subset G\subset U_n^+$, i.e. some definition of $U_n^*$ in analogy to $O_n^*$.

In the orthogonal case, the classification has been completed \cite{raum2013full} showing that the class of easy quantum groups roughly falls into three parts: the non-hyperoctahedral categories (finitely many, exactly 13), the hyperoctahedral ones not containing $\primarypart$ (one discrete series), and the hyperoctahedral ones containing $\primarypart$ (semidirect products, huge class). 
 In the unitary case, hyperoctahedral categories should be those containing $\vierpartwbwb$ but not $\singletonw\otimes\singletonb$. For the moment, we did not try to divide the unitary easy quantum groups into these three classes and to classify them along the lines of the orthogonal case. It is recommended to begin with the non-hyperoctahedral case, since this seems to be the simplest one.

\subsection{Using more colors}

Freslon developed in  \cite{freslon2014partition} ways of assigning quantum groups to categories of partitions involving $n$ colors and $n$ inverse colors (thus, our case would be $n=1$). A classification of such categories is completely open. Even in the noncrossing case, nothing is known, and it would be interesting to see if the step from $n=1$ to $n>1$ is substantial or not.

\subsection{Iterating free and tensor complexification}

From Theorems \ref{ThmClassiGlob} and \ref{ThmClassiLoc}, we deduce that all free easy quantum groups may be constructed from certain base quantum groups and the product constructions from Sections \ref{SectProd1}, \ref{SectProd2} and \ref{SectProd3}. Concerning these base quantum groups, recall that we have the following isomorphisms:
\[H_n^+\cong \Z_2\wr_* S_n^+,\quad 
B_n^{\#+}\cong O_{n-1}^+*\widehat{\Z_2},\quad
B_n^+\cong O_{n-1}^+,\quad
C_n^+\cong U_{n-1}^+\]
The first isomorphism can be found in Section \ref{SectRefl} while the second and the third are proven in \cite[Thm. 4.1]{raum} (note that Raum uses the notation $B_n'^+$ for the quantum group $B_n^{\#+}$ since at the time it was unknown that there are in fact \emph{two} quantum versions of the group $B_n'$, see \cite[Rem. 2.4 and Sect. 5]{weber2013classification} for a discussion). The proof of $C_n^+\cong U_{n-1}^+$ is verbatim the same as for $B_n^+\cong O_{n-1}^+$. The latter three isomorphisms are of the kind as in Definition \ref{DefW}(b).

We infer that all free easy quantum groups may be obtained from $S_n^+$, $O_n^+$, $\widehat{\Z_k}\wr_* S_n^+$ and the ($r$-selfadjoint) free and tensor complexifications together with isomorphisms of the type $B_n^+\cong O_{n-1}^+$. What about the converse?

\begin{quest}\label{question}
Iterating these constructions -- do we always end up with a (free) easy quantum group? 
\end{quest}

Both kind of answers would be very interesting, an affirmative one and a negative one. In the first case, we would have an alternative description of the class of free easy quantum groups in an inductive way. In the second case, we would have a machine to produce non-easy quantum groups which is not yet available for the moment (recall that by Lemma \ref{LemNonEasyMaybe}, we always have $S_n\subset G\tensorglued\widehat{\Z_d}\subset U_n^+$ and $S_n\subset G\freeglued\widehat{\Z_d}\subset U_n^+$ for $S_n\subset G\subset U_n^+$).
This question certainly needs further investigation in the future. However, we can already shed some light on it by giving a partial result on the iteration process.

\begin{prop}
Let $G$ be a compact matrix quantum group. Let $k,l\in\N_0\backslash\{1\}$ and let $m$ be the least common multiple of $k$ and $l$ (with $m=0$ if $k=0$ or $l=0$). We then have (as an identification in the sense of Def. \ref{DefW}(a)):
\[(G\tensorglued \widehat{\Z_k})\tensorglued \widehat{\Z_l}=G\tensorglued \widehat{\Z_m}\]
Again, we use the convention $\Z_0=\Z$. 
\end{prop}
\begin{proof}
Let $z_k$, $z_l$, $z_m$ denote the generators of $C^*(\Z_k)$, $C^*(\Z_l)$ and $C^*(\Z_m)$ respectively. Let $\alpha$ be the map:
\begin{align*}
\alpha: C(G)\otimes C^*(\Z_m)&\to (C(G)\otimes C^*(\Z_k))\otimes C^*(\Z_l)\\
u_{ij}&\mapsto u_{ij}\\
z_m&\mapsto z_kz_l
\end{align*}
This homomorphism exists by the universal property. Here, we used that $m$ is a multiple of $k$ and of $l$ respectively in order to prove $(z_kz_l)^m=1$. Now, the restriction of $\alpha$ to $C^*(u_{ij}z_m)$ is surjective onto $C^*(u_{ij}z_kz_l)$. It remains to prove injectivity of $\alpha$.

Consider first the homomorphism  $\alpha_1:C^*(\Z_m)\to C^*(\Z_k)\otimes C^*(\Z_l)$ given by $\alpha_1(z_m):=z_kz_l$. We convince ourselves that the elements  $(z_kz_l)^t$ for $0\leq t<m$ are linearly independent. This can be seen by representing $C^*(\Z_k)$ on $M_k(\C)$ by sending $z_k$ to the $k$-cyclic shift operator $s_k$ (i.e. $s_ke_n=e_{n+1}, s_ke_k=e_1$); and likewise mapping $z_l$ to $s_l$ in $M_l(\C)$. Now, applying $(s_k\otimes s_l)^t$ to $e_1\otimes e_1$ yields $m$ linearly independent vectors. Therefore, a dimension count proves that $\alpha_1$ is injective. Since also the restriction $\alpha_2$ of $\alpha$ to $C(G)$ is injective, we obtain that $\alpha=\alpha_1\otimes\alpha_2$ is injective (see for instance \cite[Prop. 3.1.12, Exer. 3.4.1]{brownozawa}).
\end{proof}

Using techniques from free probability, we may also prove (see \cite{tarragoweberweingar}):
\[(G\freeglued \widehat{\Z_k})\freeglued \widehat{\Z_l}=(G\tensorglued \widehat{\Z_k})\freeglued \widehat{\Z_l}=G\freeglued \widehat{\Z}\]
It is quite likely, that there is also a direct algebraic proof (similar to the one of the proposition above), but we do not have one at hand at the moment. In any case, an iteration of the free and the tensor complexifications should always yield:
\[G\tensorglued\widehat{\Z_k},\qquad
G\freeglued\widehat{\Z_k},\qquad\textnormal{or}\qquad
(G\freeglued\widehat{\Z_k})\tensorglued\widehat{\Z_l}\]
It remains to analyze the behavior with respect to the $r$-selfadjoint complexifications and the isomorphisms of the kind $B_n^+\cong O_{n-1}^+$ (see Def. \ref{DefW}(b)) in order to tackle Question \ref{question}. Moreover, it is straightforward to define an $r$-selfadjoint tensor complexification, too. Is it reasonable or does it boil down to the ordinary tensor complexification? We can also ask for a refinement of Proposition \ref{OnIsom}: How does $O_n^+\freeglued\widehat{\Z_k}$ look like? Is it isomorphic to $O_n^+\freeglued\widehat{\Z}$? Do we have a general statement of the form $G\freeglued\widehat{\Z_k}\cong G\freeglued\widehat{\Z}$ whenever the fundamental representation $u$ of $G$ is irreducible? How does $S_n^+ \freeglued_r \widehat{\Z_k}$ or $O_n^+ \freeglued_r \widehat{\Z_k}$ look like, are they easy? What does $(\widehat{\Z_d}\wr_* S_n^+)\freeglued\widehat{\Z_k}$ yield? Plenty of questions.

\subsection{Use of easy quantum groups for the theory of quantum groups}

It is the basic philosophy of easy quantum groups that they should form an accessible class of compact matrix quantum groups, simply because of their basic combinatorial feature. In this sense, general theorems about compact matrix quantum groups could be found by first investigating the easy quantum groups. An example for that method can be found in \cite{raum2013easy}, where \emph{any} quantum group $S_n\subset G\subset O_n^+$ with $u_{ij}^2$ being central projections can be understood as a semi-direct product. This fact has first been observed in the case of easy quantum groups, before it has been extend also to non-easy quantum groups. Another example is the description of some partial fusion rules for all easy quantum groups in a uniform way \cite{freslonweber} (which turn out to be the honest fusion rules in the case of categories of noncrossing partitions). By that, we obtain insight in some parallelism between $S_n^+$ and $O_n^+$, for instance. The hope is that more theorems of this kind follow, and that it turns out that easy quantum groups are somehow the combinatorial backbone of all compact matrix quantum groups (at least for all of the type $S_n\subset G\subset U_n^+$).

\subsection{Use of easy quantum groups for free probability}

Quantum groups seem to be the right notion of symmetries for free probability (see \cite{nicaspeicher} for an introduction to the latter subject). The most important applications are de Finetti theorems, \cite{kostler2009noncommutative} and \cite{banica2012finetti}. For $U_n^+$ there is a de Finetti theorem by Curran, see \cite{curranQRota}, but for other unitary easy quantum groups nothing is known.
Furthermore, the laws of characters appearing with easy quantum groups are well known players in free probability. In \cite{banica2011stochastic}, Banica, Curran and Speicher extended the results of Diaconis and Shahshahani \cite{diaconis1994eigenvalues} on the law of traces of powers of $O_{n}, U_{n}$ and $S_{n}$ to all orthogonal free easy quantum groups. In a forthcoming paper \cite{tarragoweberweingar}, the authors will extend these results to all free easy quantum groups. See also \cite{brannan} for other applications of easy quantum groups to free probability.

\subsection{Links with quantum isometry groups}

Quantum isometry groups are a powerful machinery in order to provide examples of quantum groups. The very nice feature is, that they immediatly come with a (noncommutative) space on which they act. Thus, noncommutative geometry and quantum groups come together in a  twinned way. Several examples of easy quantum groups can be seen as quantum isometry groups, see  \cite{wang1998quantum}, \cite{banica2010quantum}, \cite{banica2015liberation}. The question is whether \emph{all} easy quantum groups appear to be quantum isometry groups; and whether conversely all quantum isometry groups $S_n\subset G\subset U_n^+$ are easy. It would be fruitful for both sides to have these two kinds of quantum group schools intervowen and maybe to find a way of describing the natural noncommutative manifolds arising with such ``easy'' quantum isometry groups in a purely combinatorial way. 

\subsection{Links with Woronowicz's $SU_q(n)$}

In \cite{woronowicz1988tannaka}, Woronowicz considers unital universal $C^*$-algebras $A$ generated by elements $u_{ij}$, $1\leq i,j\leq n$ such that $u=(u_{ij})$ is unitary and
\[\sum_{j_1,\ldots,j_n} E(j_1,\ldots,j_n) u_{i_1j_1}\ldots u_{i_nj_n}=E(i_1,\ldots,i_n)\]
where $E(i_1,\ldots,i_n)$ are complex numbers. For $k\in\{1,\ldots,n\}$ he puts:
\[E_{k-}:=\left(E(k,i_2,\ldots,i_n)\right)_{1\leq i_2,\ldots,i_n\leq n},\qquad
E_{-k}:=\left(E(i_1,\ldots,i_{n-1},k)\right)_{1\leq i_1,\ldots,i_{n-1}\leq n}\]
He proves that $(A,u)$ is a compact matrix quantum group, if $E_{1-},\ldots, E_{n-}$ are linear independent and $E_{-1},\ldots, E_{-n}$, too. Putting
\[E(i_1,\ldots,i_n):=\begin{cases} 0&\textnormal{ if }i_k=i_l \textnormal{ for some } k,l\\
 (-q)^{I(i_1,\ldots,i_n)} &\textnormal{ otherwise}\end{cases}\]
 where $I(i_1,\ldots,i_n)$ is the minimal number of transpositions needed to turn the string $(i_1,\ldots,i_n)$ into $(1,2,\ldots,n)$, he defines his quantum group $SU_q(n)$ for $q\in (0,1]$ with $q=1$ being the classical case $SU(n)$.
 
 Let $n\geq 3$. Our quantum groups $\widehat{\Z_n}\wr_* S_n^+$ corresponding to
 \[\categ{\HHH}{\loc}{n,n}= \categ{\HHH}{\loc}{n,0}=\langle b_n,\vierpartwwbb,\vierpartwbwb\rangle=\langle b_n\rangle\]
 (apply \cite[Lemma 1.1(b)]{tarragowebercombina} to $b_n\otimes \tilde b_n$ and use \cite[Lemma 1.3(c)]{tarragowebercombina} in order to see the last of the above equalities) are somehow complementary to $SU_q(n)$. Indeed, putting
 \[E(i_1,\ldots,i_n):=\delta_{b_n}(0,i)=\begin{cases} 1&\textnormal{ if all $i_k$ coincide}\\
 0 &\textnormal{ otherwise}\end{cases}\]
 we obtain $\widehat{\Z_n}\wr_* S_n^+$ from Woronowicz's definition.
 
\subsection{Links with dual groups}
 
In \cite{voicdual}, Voiculescu introduced dual groups as unital $C^*$-algebras $A$ together with $*$-homomorphisms $\Delta:A\to A*A$, $S:A\to A$ and $\epsilon:A\to \C$ such that natural Hopf algebraic like relations between these maps are fulfilled. These objects also appear under different names (such as co-groups) in other contexts. The crucial difference to quantum groups is that their comultiplication $\Delta$ goes into the free product of the $C^*$-algebras rather than into their tensor product. In this sense, dual groups are even more ``quantum'' (in the sense of noncommutativity) than quantum groups, but the prize is that they do not generalize the notion of groups. Indeed, the multiplication $\mu:G\times G\to G$ of a group translates into $\Delta:C(G)\to C(G\times G)$ on the algebraic level, but we have $C(G\times G)\cong C(G)\otimes C(G)$ rather than $C(G\times G)\cong C(G)*C(G)$.

However, dual groups have been studied as reasonable noncommutative analogues of groups in various settings, see for instance \cite{cebronulrich} for an overview on the literature. The main example of a dual group is Brown's algebra \cite{brown}  denoted by $U\langle n\rangle$ or $U_n^{nc}$. In the language of $C^*$-algebras it is given by the unital universal $C^*$-algebra $C(U\langle n\rangle)$ generated by elements $u_{ij}$, $1\leq i,j\leq n$ such that $u=(u_{ij})$ is unitary, together with the maps:
\[\Delta(u_{ij})=\sum_k u_{ik}^{(1)}u_{kj}^{(2)},\qquad S(u_{ij})=u_{ji}^*, \qquad \epsilon(u_{ij})=\delta_{ij}\]
Here, $u_{ik}^{(1)}$ lies in the first copy of $A$ in $A*A$ while $u_{kj}^{(2)}$ is in the second. Note that we do not require that $\bar u=(u_{ij}^*)$ is unitary and in fact this relation does not pass the comultiplication $\Delta$. Thus, $U\langle n\rangle$ is not a compact matrix quantum group, while $U_n^+$, $O_n^+$ and $S_n^+$ are no dual groups.

The pool of examples of quantum subgroups of $U_n^+$ is huge -- amongst others thanks to the easy quantum group machine. Finding dual subgroups $G$ of $U\langle n\rangle$ (in the sense that there is a $*$-homomorphism from $C(U\langle n\rangle)$ to $C(G)$ mapping generators to generators) in turn is a much more difficult task -- there are not so many, and in particular, we have no ``easy'' approach, using intertwiners and partitions. However, from this article we obtain two new examples of dual groups $B\langle n\rangle\subset B'\langle n\rangle\subset U\langle n\rangle$:
\begin{align*}
C(B\langle n\rangle):=&C^*(u_{ij}, 1\leq i,j\leq n\;|\; u\textnormal{ is unitary and the relations } R(\singletonw)\textnormal{ are fulfilled})\\
C(B'\langle n\rangle):=&C^*(u_{ij}, 1\leq i,j\leq n\;|\; u\textnormal{ is unitary and the relations } R(\singletonw\otimes\singletonb)\textnormal{ are fulfilled})
\end{align*}
The maps $\Delta$, $S$ and $\epsilon$ are as in the case of $U\langle n\rangle$. All other relations $R(p)$ of Section \ref{SectCStar} do not pass the comultiplication.

\bibliographystyle{alpha}
\nocite{*}
\bibliography{bibliographieOpAlg}

\end{document}